\documentclass[12pt]{amsart}
\usepackage{amsfonts,amssymb,amscd,amsmath,enumerate,verbatim}
\usepackage[utf8]{inputenc}
\usepackage{latexsym}
\usepackage{mathptmx}
\usepackage{color}
\usepackage{multicol,longtable}
\input xy
\xyoption{all}
\usepackage[left]{lineno}
\usepackage{aliascnt}
\usepackage[capitalise, noabbrev]{cleveref}
\usepackage{tikz}
\definecolor{verylight}{gray}{0.97}
\definecolor{light}{gray}{0.9}
\definecolor{medium}{gray}{0.85}
\definecolor{dark}{gray}{0.6}

\newtheorem{theorem}{Theorem}[section]
\newtheorem{maintheorem}{Theorem}

\crefname{maintheorem}{Theorem}{Theorems}
\Crefname{maintheorem}{Theorem}{Theorems}
\newaliascnt{lemma}{theorem}
\newtheorem{lemma}[lemma]{Lemma}
\aliascntresetthe{lemma}
\newaliascnt{proposition}{theorem}

\aliascntresetthe{proposition}
\newaliascnt{corollary}{theorem}
\newtheorem{corollary}[corollary]{Corollary}
\aliascntresetthe{corollary}
\newaliascnt{conjecture}{theorem}

\aliascntresetthe{conjecture}
\newaliascnt{construction}{theorem}

\aliascntresetthe{construction}
\newaliascnt{claim}{theorem}

\aliascntresetthe{claim}
\newaliascnt{addendum}{theorem}

\aliascntresetthe{addendum}
\newaliascnt{question}{theorem}
\newtheorem{question}[question]{Question}
\aliascntresetthe{question}

\theoremstyle{definition}
\newaliascnt{definition}{theorem}

\aliascntresetthe{definition}
\newaliascnt{remark}{theorem}
\newtheorem{remark}[remark]{Remark}
\aliascntresetthe{remark}
\newaliascnt{example}{theorem}

\aliascntresetthe{example}

\newaliascnt{notation}{theorem}
\newtheorem{notation}[notation]{Notation}
\aliascntresetthe{notation}
\newaliascnt{discussion}{theorem}

\aliascntresetthe{discussion}
\newaliascnt{setting}{theorem}
\newtheorem{setting}[setting]{Setting}
\aliascntresetthe{setting}
\newaliascnt{ordering}{theorem}
\newtheorem{ordering}[ordering]{Ordering}
\aliascntresetthe{ordering}

\newcommand{\reg}{\mathrm{reg}}
\def\GG{{\mathcal G}}
\def\SS{{\mathcal S}}

\begin{document}

\title{Edge ideals of gapfree graphs with matching number $3$}

\author[N.~Erey]{Nursel Erey}
\author[T.~H.~H\`a]{T\`ai Huy H\`a}
\author[T.~Hibi]{Takayuki Hibi}
\author[S.~Kara]{Selvi Kara}
\author[S.~Morey]{Susan Morey}

\address{(Nursel Erey) Gebze Technical University, Department of Mathematics, 41400 Gebze, Kocaeli, Turkey}
\email{nurselerey@gtu.edu.tr}
\address{(T\`ai Huy H\`a) Mathematics Department, Tulane University, 6823 St. Charles Avenue, New Orleans, LA 70118, USA}
\email{tha@tulane.edu}
\address{(Takayuki Hibi) Department of Pure and Applied Mathematics, Graduate School of Information Science and Technology, Osaka University, Suita, Osaka 565--0871, Japan}
\email{hibi@math.sci.osaka-u.ac.jp}
\address{(Selvi Kara) Department of Mathematics, Bryn Mawr College, Bryn Mawr, PA 19010, USA}
\email{skara@brynmawr.edu}
\address{(Susan Morey) Department of Mathematics, Texas State University, 601 University Dr., San Marcos, TX 78666, USA}
\email{morey@txstate.edu}



\begin{abstract}
We prove that if $G$ is a gapfree graph with matching number $3$, then every power $I(G)^q$, $q\geq2$, of its edge ideal has linear quotients, and hence a linear resolution. Our approach combines graph reductions based on clique domination of gapfree graphs and matching number with operations preserving linear quotients. We reduce the problem to $26$ non-co-chordal graphs on at most seven vertices. A new preservation theorem for gapfree edge addition, together with known results on duplication and gapfree expansion, handles all but one of these graphs. For the remaining graph on seven vertices, we construct a linear quotients ordering for the square and extend it recursively to every higher power.
\end{abstract}

\maketitle

\section*{Introduction}
Minimal graded free resolutions describe the generators of an ideal, the relations among them, and the higher relations among those relations. An ideal generated in degree $d$ has a \emph{linear resolution} if its $i$th syzygies occur only in degree $d+i$. Equivalently, its Castelnuovo--Mumford regularity is $d$, the smallest possible value. Identifying ideals with linear resolutions therefore amounts to finding situations in which the degrees of all syzygies are controlled by the generating degree. For edge ideals, this problem also asks how the combinatorics of a graph controls its algebraic invariants.

Let $G$ be a finite simple graph with vertex set $V(G)=\{x_1,\ldots,x_n\}$ and at least one edge, and let $S=K[x_1,\ldots,x_n]$ be a polynomial ring over a field $K$. The \emph{edge ideal} of $G$ is
\[
I(G)=(x_ix_j:\{x_i,x_j\}\in E(G))\subseteq S.
\]
Its $q$th power $I(G)^q$  is generated in degree $2q$. It has a linear resolution exactly when $\reg~ I(G)^q=2q$. Independent results of Cutkosky, Herzog and Trung \cite{CHT} and Kodiyalam \cite{Kodiyalam} imply that $\reg~ I(G)^q=2q+b$ for all sufficiently large $q$, where $b\geq 0$ is a constant. These results establish eventual linear growth of regularity, but determining when $b=0$ and when linear resolutions begin requires further information about $G$.

The case in which the first power has a linear resolution is completely understood. Fr\"oberg's theorem \cite{F} says that $I(G)$ has a linear resolution if and only if the complement of $G$ is chordal, meaning that it has no induced cycle of length at least $4$. Such a graph $G$ is called \emph{co-chordal}. Herzog, Hibi and Zheng \cite{HHZ} showed that in this case every power has a linear resolution. In fact, every power has \emph{linear quotients}: its minimal monomial generators can be ordered $u_1,\ldots,u_m$ so that each ideal $(u_1,\ldots,u_{j-1}):u_j$, for $j\geq 2$, is generated by variables. For an ideal generated in one degree, linear quotients imply a linear resolution.

The question becomes more subtle when $I(G)$ itself does not have a linear resolution. A necessary condition for a power of an edge ideal to have a linear resolution is that $G$ be \emph{gapfree} \cite{NP}, that is, that $G$ contains no induced subgraph consisting of two disjoint edges. Equivalently, any two disjoint edges are joined by another edge. Algebraically, Bigdeli, Herzog and Zaare-Nahandi \cite{BHZ} showed that gapfreeness is equivalent to every  power of $I(G)$ being linearly presented: the first syzygies of $I(G)^q$ all have degree $2q+1$. The issue is thus whether linearity of these first syzygies eventually extends to the entire resolution.

Nevo and Peeva in \cite{NP} constructed a gapfree graph whose edge ideal has a nonlinear square, showing that gapfreeness alone does not force linear resolutions to begin at the second power. They conjectured that it does force $I(G)^q$ to have a linear resolution for all sufficiently large $q$ \cite[Open Problem~1.11(1)]{NP}. Their stronger regularity conjecture  \cite[Open Problem~1.11(2)]{NP} would imply linear resolutions for every $q\geq\reg~ I(G)-1$. In particular, this conjecture suggests that a gapfree graph with $\reg~ I(G)=3$ has linear resolutions for all powers starting with the square.

We approach this question through a combinatorial restriction. A \emph{matching} is a set of pairwise vertex disjoint edges, and the \emph{matching number} $\nu(G)$ is the maximum size of a matching in $G$. Our main result is the following.

\begin{maintheorem}
\label{t:MAINtheorem}
Let $G$ be a gapfree graph with matching number $3$. Then $I(G)^q$ has linear quotients for all $q\geq 2$.
\end{maintheorem}

Previous work on the Nevo--Peeva conjecture has sought algebraic hypotheses guaranteeing linear resolutions for all powers starting with the square.  Banerjee \cite{Banerjee} developed an inductive method that bounds the regularity of successive powers using colon ideals $(I(G)^{q+1}:m)$, where $m$ is a minimal generator of $I(G)^q$. Building on this method, Banerjee, Kara and H\`a \cite{BKH} proved that every power starting with the square has a linear resolution when $G$ is gapfree and \emph{locally linear}, meaning that $\reg~(I(G):x)\leq 2$ for every vertex $x$. Local linearity implies $\reg~ I(G)\leq 3$. Results from \cite{Bigdeli,BHZ} require the first nonlinear syzygies of the original edge ideal $I(G)$ to occur at the last or penultimate step of its minimal free resolution; for gapfree graphs, this also gives linear resolutions for every power starting with the square. These conditions control colon ideals or the pattern of syzygies, but it can be difficult to recognize which classes of graphs satisfy them directly from their combinatorics. This motivates studying conditions that can be read from the graph itself.

Combinatorial restrictions have also proved effective. Banerjee established linear resolutions for all powers $q\geq 2$ of gapfree, cricket-free graphs \cite{Banerjee}. Erey obtained the same conclusion for gapfree, diamond-free graphs \cite{Erey2018} and for gapfree, $C_4$-free graphs \cite{Erey2019}. Here each additional condition excludes the indicated graph as an induced subgraph. These results give substantial families for which graph structure controls every higher power.

In \cite{MV}, Minh and Vu determined exactly how gapfreeness and regularity control the square and cube of edge ideals:
\[
\begin{aligned}
I(G)^2\text{ has a linear resolution}
&\quad\Longleftrightarrow\quad
G\text{ is gapfree and }\reg~ I(G)\leq 3,\\
I(G)^3\text{ has a linear resolution}
&\quad\Longleftrightarrow\quad
G\text{ is gapfree and }\reg~ I(G)\leq 4.
\end{aligned}
\]
These characterizations motivate the use of matching number in Theorem \ref{t:MAINtheorem}, since it bounds regularity and provides structure for graph reductions.

A matching is \emph{induced} if the induced graph on the vertices of the matching contains no additional edges. Write $\nu_{\mathrm{ind}}(G)$ for the maximum size of an induced matching. The two invariants, matching number and induced matching number, provide bounds for the regularity: $ \nu_{\mathrm{ind}}(G)+1\leq\reg~ I(G)\leq\nu(G)+1. $ The lower bound is due to Katzman \cite{Katzman}, and the upper bound is due to H\`a and Van Tuyl \cite{HaVanTuyl}. In   \cite{TNT}, Trung characterized equality in the upper bound: for a connected graph $G$, one has $\reg~ I(G)=\nu(G)+1$ if and only if $G$ is a pentagon or $\nu_{\mathrm{ind}}(G)=\nu(G)$. This characterization allows us to sharpen the bound using the structure of the graph.

For gapfree graphs, one has $\nu_{\mathrm{ind}}(G)=1$ and the small matching cases are already understood.  When $\nu(G)=1$, the bounds force $\reg~ I(G)=2$. Thus the complement of $G$ is chordal, and every power of its edge ideal has linear quotients. When $\nu(G)=2$, the upper bound gives $\reg~ I(G)\leq3$. So Minh and Vu's results already provide linear resolutions for the square and cube. In fact, all powers starting with the square have linear quotients: by Lemma \ref{lem:H-structure}, either the complement of $G$ is chordal or its only component containing an edge is a pentagon. The pentagon case was established in our earlier work \cite[Proposition~4.2]{EFHHKM}.

Matching number three is therefore the next case. The upper bound gives $\reg~ I(G)\leq4$, but neither equality case in Trung's characterization is possible: the pentagon has matching number two, while $\nu_{\mathrm{ind}}(G)=1<3=\nu(G)$. Hence $\reg~ I(G)\leq3$, and Minh and Vu's results give linear resolutions for the square and cube. Theorem \ref{t:MAINtheorem} extends this conclusion to every higher power and gives linear quotients for all powers starting with the square.


Our approach in this paper continues the study in \cite{EFHHKM}, where we investigated linear quotients as a way to address the Nevo--Peeva conjecture. In that work, we proposed that once a power of an edge ideal has linear quotients, all subsequent powers should have them, and studied graph operations preserving this property. In particular, duplication and gapfree expansion preserve linear quotients at each fixed power. These operations produced families beyond the preceding forbidden subgraph classes, including graphs containing induced crickets, diamonds, $C_4$ and $C_5$. In this paper,  the preservation results serve as reduction tools: we remove vertices that can later be restored by these operations, and use the matching restriction to control the graphs that remain.

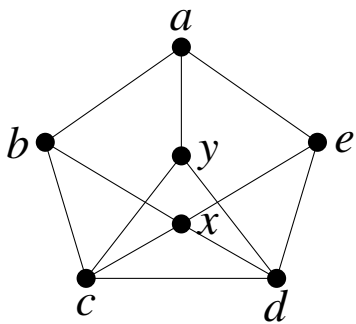
\begin{figure}[ht]
\centering
\begin{tikzpicture}[scale=1.8]
\coordinate (z) at (0,1){};
\coordinate (p) at (-1,0.3){};
\coordinate (q) at (1,0.3){};
\coordinate (x) at (0,0.2){};
\coordinate (a) at (-0.7,-0.7){};
\coordinate (b) at (0.7,-0.7){};
\coordinate (w) at (0,-0.3){};
\fill(z)circle(0.7mm);
\fill(p)circle(0.7mm);
\fill(q)circle(0.7mm);
\fill(x)circle(0.7mm);
\fill(a)circle(0.7mm);
\fill(b)circle(0.7mm);
\fill(w)circle(0.7mm);
\draw(z)--(p)--(a)--(b)--(q)--cycle;
\draw(z)--(x)--(a);
\draw(x)--(b);
\draw(q)--(w)--(b);
\draw(p)--(w)--(a);
\draw(0,1.2)node{{\Large $a$}};
\draw(-1.2,0.3)node{{\Large $b$}};
\draw(1.2,0.3)node{{\Large $e$}};
\draw(0.2,0.2)node{{\Large $y$}};
\draw(-0.7,-0.9)node{{\Large $c$}};
\draw(0.7,-0.9)node{{\Large $d$}};
\draw(0.2,-0.3)node{{\Large $x$}};
\end{tikzpicture}
\caption{The remaining graph $G_{12}$.}\label{fig:exceptional}
\end{figure}

Section \ref{sec:edge-addition} establishes a new preservation result for adding an edge, including attaching a leaf, when the resulting graph remains gapfree (Lemma \ref{Nursel_Susan}). In Section \ref{sec:domination}, these reductions allow us to focus on  non-co-chordal gapfree graphs with matching number $3$ and without duplicate vertices, parallel vertices or whiskers. Combining the matching restriction with the domination theorem of El Zahar and Erd\H{o}s \cite[Theorem~5]{EZE} (see also \cite{BT1990}), we show that such graphs have at most seven vertices. The resulting finite list consists of $26$ graphs (Lemma \ref{lem.26}). Appendix \ref{app:code} gives the Python code used to create the list, and Appendix \ref{app:graphs} displays the graphs. Section \ref{sec:finite-reduction} uses the preservation results and known cases to reduce this list to the graph in Figure \ref{fig:exceptional} (Lemma \ref{pending}). Finally, in Section \ref{sec:main-proof}, we construct a linear quotients ordering for every power $I(G)^q$, $q\geq 2$, of the graph in Figure \ref{fig:exceptional} (Lemma \ref{prop:fig1_all_powers_lq}), completing the proof of Theorem \ref{t:MAINtheorem}.


\section{Adding a new edge to a gapfree graph}\label{sec:edge-addition}

We write $N_G(x)$ for the open neighborhood of $x$ and $N_G[x]=N_G(x)\cup\{x\}$ for its closed neighborhood. Distinct vertices $x,y$ are \emph{duplicates} if $N_G(x)=N_G(y)$ and \emph{parallel} if $N_G[x]=N_G[y]$. A \emph{leaf} is a vertex of degree one; we use ``a whisker attached to $x$'' to refer to a leaf whose unique neighbor is $x$.

The \emph{duplication} of $G$ at a vertex $x$ adds a new vertex $y$ with the same open neighborhood as $x$. Denote the resulting graph by $G^x$. The \emph{expansion} of $G$ at $x$, denoted by $G^{[x]}$, also adds the edge $\{x,y\}$; thus $x$ and $y$ have the same closed neighborhoods. We also refer to this second operation as \emph{parallelization}. Explicitly,
\[
\begin{aligned}
E(G^x)&=E(G)\cup\{\{y,z\}:\{x,z\}\in E(G)\},\\
E(G^{[x]})&=E(G^x)\cup\{\{x,y\}\}.
\end{aligned}
\]
We recall the following preservation results from our earlier work. Here we assume $G$ is a graph and $q$ is a positive integer.

\begin{lemma}[{\cite[Proposition~2.4]{EFHHKM}}]
\label{duplication}
If $I(G)^q$ has linear quotients, then $I(G^x)^q$ has linear quotients.
\end{lemma}

\begin{lemma}[{\cite[Theorem~3.4]{EFHHKM}}]
\label{expansion}
If $I(G)^q$ has linear quotients and if $G^{[x]}$ is gapfree, then $I(G^{[x]})^q$ has linear quotients.
\end{lemma}

Our new edge addition result allows an edge to be added between existing vertices or a leaf to be attached to a vertex, provided that gapfreeness is preserved.

\begin{corollary}\label{Nursel_Susan}
Let $H$ be a gapfree graph and $y\in V(H)$. Suppose that either $x$ is a new vertex or that $x\in V(H)$ with $\{x,y\}\notin E(H)$. Let $G$ be the graph with $V(G)=V(H)\cup\{x\}$ and $E(G)=E(H)\cup\{\{x,y\}\}$. If $G$ is gapfree and $I(H)^q$ has linear quotients for all $q\geq 2$, then $I(G)^q$ has linear quotients for all $q\geq 2$.
\end{corollary}

We prove Corollary  \ref{Nursel_Susan} by establishing Theorem \ref{t:add edge all powers}. We first introduce notation and an ordering of the generators of $I(G)$. For a monomial ideal $J$, let $\GG(J)$ denote its set of minimal monomial generators. For monomials $u,v$, we write $u:v=u/\gcd(u,v)$, also denoted by $(u:v)$.

Let $I(G)$ be the edge ideal of a finite graph $G$ and $e \in E(G)$.  We write $e=xy$ for both the edge $\{x,y\}$ and the monomial generator $xy$ of $I(G)$. We use the notations $\shortmid$ and $\mid$ respectively for \textit{edge division} and \textit{monomial division}. If $u$ is a monomial in $I(G)^t, t \geq 1$, we write $e \shortmid u$ if $u = e \cdot e_2 \cdots \cdot e_t$ with each $e_i \in E(G)$.  In other words, $u$ can be written as a product of $t$ edges, one of which is $e$.  For example, if $G$ is the square on $\{x_1,x_2,x_3,x_4\}$ with the edges $e_1=x_1x_2, e_2=x_2x_3, e_3=x_3x_4, e_4=x_4x_1$ and if $u=e_2e_4=x_1x_2x_3x_4$, then $e_1 \shortmid u$ since $u=e_1e_3$ as well, while $x_1x_3 \mid u$ but $x_1x_3 \nshortmid u$.

\begin{setting}\label{setting}
Let $H$ be a gapfree graph and let $y \in V(H)$. Assume that either $x$ is a new vertex or that $x \in V(H)$ with $\{x,y\} \not\in E(H)$. Let $G$ be the graph with $V(G)=V(H) \cup \{x\}$ and $E(G) = E(H) \cup \{\{x,y\}\}$, which is obtained from $H$ by adding a new edge $\{x,y\}$.
\end{setting}

The ordering below groups generators according to the minimum number of copies of the new edge in a factorization. The preliminary lemmas describe these groups and the colon ideals needed for linear quotients. We first handle the square and then pass to higher powers by induction.

\begin{ordering}\label{ordering}
Assume $H$, $x$, $y$, and $G$ are as in \cref{setting}. Let $t \ge 2$ and assume that $I(H)^q$ has a linear quotients ordering for $2 \le q \le t$. In order to define an ordering on the generators of $I(G)^t$, we define the $t+1$ disjoint sets
$$ S_1= \GG(I(H)^t), \quad S_2 =\GG(I(H)^{t-1}(xy))\setminus S_1,\quad S_3=\GG(I(H)^{t-2}(xy)^2)\setminus \bigcup_{i=1}^2 S_i, \ldots, $$
$$ S_{t} =\GG(I(H)(xy)^{t-1}) \setminus \bigcup_{i=1}^{t-1}S_i, \quad S_{t+1}=\{(xy)^t\}.$$
We  say $S_1,\dots, S_{t+1}$ are the \textit{$S$-sets of $I(G)^t$}. Set $s_0=0, s_1=|S_1|$ and $s_i=\sum_{j=1}^i |S_j|$.

Order the generators
$$u_1<\dots <u_{s_1}<u_{s_1+1}<\dots <u_{s_2}< u_{s_2+1}< \ldots < u_{s_3}< \dots < (xy)^t$$
of $I(G)^t$ so that when $i\le t-1$ we have $u_{s_{i-1}+1},\dots ,u_{s_{i}}$ is a linear quotients ordering for $S_{i}$ formed by starting with a linear quotients ordering of $I(H)^{t-i+1}$, multiplying each element by $(xy)^{i-1}$, and removing any elements already ordered in $S_j$ for $j <i$.

The elements of $S_t$ are of the form $e(xy)^{t-1}$ where $e\in I(H)$. Order these elements so all elements with $e=ab$ where $a,b \ne x,y$ come before all elements with $e=ay$ or $e=ax$ (note that $e=xy$ is not possible since $xy \not\in E(H)$).

Notice that under this order, all elements of $S_i$ come before all elements of $S_j$ for $1 \le i < j \le t+1$.

Moreover, given a generator $u$ of $\GG(I(G)^t)$, we have $u\in S_i$ where $i-1$ is the minimum number of copies of $xy$ used in any representation of $u$ as a product of $t$ edges of $G$. We  call such a representation of $u$ an {\it optimal factorization} of $u$.
\end{ordering}

The next lemma records how multiplication and division by an edge affect the $S$-sets in \cref{ordering}. When $x\in V(H)$, multiplying by an edge can decrease the number of copies of $xy$ needed in a factorization.

\begin{lemma}\label{l:reduce xy}
Assume $H$, $x$, $y$, and $G$ are as in \cref{setting}. Let $t \ge 3$ and assume that $I(H)^q$ has a linear quotients ordering for $2 \le q \le t$. Let $\SS_1,\dots, \SS_{t}$ be the $S$-sets of $I(G)^{t-1}$. Let $S_1,\dots, S_{t+1}$ be the $S$-sets of $I(G)^{t}$. For any generator $u$ of $I(G)^{t}$, the following statements hold.
\begin{enumerate}
    \item If $u\in S_r$ for some $2\le r\le t+1$, then $xy\shortmid u$ and $\frac{u}{xy}\in \SS_{r-1}$. Conversely, if $v\in \SS_{r-1}$, then $(xy)\cdot v \in S_s$ for some $s\le r$.
    \item Let $1\leq r\leq t$.
    \begin{enumerate}
        \item If $u=e_1\dots e_{t-r+1}(xy)^{r-1}\in S_r$ with $e_1,\dots ,e_{t-r+1}\in E(H)$, then $\frac{u}{e_1}\in \SS_r$.
        \item If $e\neq xy$ and $v \in \SS_{r}$, then $ev\in S_{r'}$ for some $r'\leq r$.
    \end{enumerate}
\end{enumerate}
\end{lemma}

We first identify variables that belong to the colon ideals for generators in $S_r$ with $r\geq 2$.

\begin{lemma}\label{l:neighbor sets}
Assume $H$, $x$, $y$, and $G$ are as in \cref{setting}. Let $t \ge 2$ and assume that $I(H)^q$ has a linear quotients ordering for $2 \le q \le t$. Then under the ordering in \cref{ordering}, if $J=((u_1, \ldots, u_{j-1}):u_j)$ for some $u_j \in S_r$ for $r \ge 2$, then
$$N_H(x) \cup N_H(y) \subseteq J.$$
In particular, for each $c \in N_H(x) \cup N_H(y)$, there exists  $u_k\in S_{\ell}$ for some $\ell < r$ such that  $u_k < u_j$ and $(u_k:u_j)=c$.
\end{lemma}

\begin{proof}
Note that since $r \ge 2$, $xy \shortmid u_j$. Let $c \in N_H(y)$ and set  $u_k=\frac{u_j(cy)}{xy}$. By \cref{ordering}, $u_k \in S_\ell$ for some $\ell<r$, so $u_k < u_j$. In particular, $(u_k:u_j) \in J$. By definition, $(u_k:u_j) = c$, so $c \in J$. If $c \in N_H(x)$ the same argument holds using $u_k=\frac{u_j(cx)}{xy}$.
\end{proof}

The next lemma handles pairs of generators in the same $S$-set, including $S_t$, whose ordering does not require $I(H)$ to have linear quotients.

\begin{lemma}\label{ordering within a set}
Assume $H$, $x$, $y$, and $G$ are as in \cref{setting}. Let $t \ge 2$ and assume that $I(H)^q$ has a linear quotients ordering for $2 \le q \le t$. Assume further that $G$ is gapfree. Then under \cref{ordering}, if $u_i < u_j$ with $u_i,u_j \in S_r$ for some $1 \le r \le t$, then there exists  $u_k \in  S_s$ for some $s \le r$ with $(u_k:u_j)$ of degree one and $(u_k:u_j) \mid (u_i:u_j)$.
\end{lemma}

\begin{proof}
Assume $1 \le r \le t-1$. Then $u_i=m(xy)^{r-1}$ and $u_j=n(xy)^{r-1}$ where $m,n \in \mathcal{G}(I(H)^{t-r+1})$. Since $r \le t-1$, we have $2 \le t-r+1 \le t$, so $I(H)^{t-r+1}$ has a linear quotients ordering. Thus there exists a monomial $w \in I(H)^{t-r+1}$ with $w < n$ and $(w:n)$ linear dividing $(m:n)$. Set $u_k=w(xy)^{r-1}$ and notice that $u_k\in S_s$ for some $s\le r$. If $s<r$, then $u_k<u_j$ since earlier $S$-sets come first. If $s=r$, then $u_k<u_j$ because the relative order inside $S_r$ is inherited from the order of $I(H)^{t-r+1}$. Since $ (u_k:u_j)=(w:n) $ and $(u_i:u_j)=(m:n)$, the result follows.

Now assume $r=t$. If $(u_i:u_j)$ is linear, the result holds with $u_k=u_i$. If $(u_i:u_j)$ is not linear, then by the definition of $S_t$, we have $(u_i:u_j)=ab$  for some $ab \in E(H)$. If $b=x$ or $b=y$, then $a \in N_H(x) \cup N_H(y)$ and the result follows from \cref{l:neighbor sets}. Otherwise, $ab$ and $xy$ are disjoint edges in $G$ and since $G$ is gapfree, either $a$ or $b$ is in $N_H(x) \cup N_H(y)$ and again the result follows from \cref{l:neighbor sets}.
\end{proof}

We begin with the second power of $I(G)$, which provides the base case for induction on the power.

\begin{lemma}\label{l:add edge square}
Assume $H$, $x$, $y$, and $G$ are as in \cref{setting}. Assume further that $G$ is gapfree. If $I(H)^2$ has linear quotients, then $I(G)^2$ has linear quotients under \cref{ordering}.
\end{lemma}

\begin{proof}
Let $u_j$ be fixed. We show that $J=((u_1,\dots ,u_{j-1}):u_j)$ is generated by variables. To do this, we show that if $u_i<u_j$, then there exists  $u_k<u_j$ such that $(u_k:u_j)$ is linear and $(u_k:u_j) \mid (u_i:u_j)$. Note that if $j \le s_1$ or equivalently if $u_j \in S_1$, the result follows from \cref{ordering within a set}. So, we assume $j \geq s_1+1$.

Assume $u_j \in S_2$ and $i < j$. If $u_i \in S_2$, then the result follows from \cref{ordering within a set}. Suppose $u_i\in S_1$ and let $u_i=(cd)(ef)$ be an optimal factorization. Since $G$ is gapfree, either $\{c,d\} \cap \{x,y\} \ne \emptyset$ or there is an edge in $G$ connecting $\{c,d\}$ to $\{x,y\}$. A similar statement holds for $\{e,f\}$. In either case, without loss of generality and by the choice of optimal factorization, we can assume $c,e \in N_H(x) \cup N_H(y) \subseteq J$ and $c,e \notin \{x,y\}$.  If $c$ or $e$ divides $(u_i:u_j)$, then the result follows from \cref{l:neighbor sets}. Therefore, we may assume $u_j=(ce)(xy)$. Then either $u_i:u_j$ is of degree one, or $d,f\notin \{x,y\}$. Since there is nothing to show in the first case, let us assume that $d,f\notin \{x,y\}$. If $c\in N_H(x)$, then $u_k=(ef)(cx)\in S_1$ is the desired generator. Similarly, if $c\in N_H(y)$, then $u_k=(ef)(cy)\in S_1$ satisfies the required condition.

Next assume $u_j \in S_3$, so $u_j = (xy)^2$. If $u_i \in S_2$ then $u_i = (cd)(xy)$ for some $\{c,d\}\in E(H)\subset E(G)$. Now since $G$ is gapfree, then without loss of generality, $c \in N_H(x) \cup N_H(y)$. Then $(u_i:u_j)$ is either linear or a multiple of $c$, in which case the result follows from \cref{l:neighbor sets}. Finally, assume $u_i \in S_1$. As above, we may assume $u_i=(cd)(ef)$ with $c,e \in N_H(x) \cup N_H(y)$. If $d$ or $f$ is $x$ or $y$, then $(u_i:u_j)$ is either linear, so $u_k=u_i$, or a multiple of $c$ or $e$, in which case the result follows from \cref{l:neighbor sets}. If $d,f \ne x,y$, then $(u_i:u_j) = cdef$, which is a multiple of $c \in N_H(x) \cup N_H(y)$ and the result again follows from \cref{l:neighbor sets}.
\end{proof}

We record a stronger conclusion from the preceding proof for use in the induction.

\begin{remark}\label{r:smaller set}
In the setting of \cref{l:add edge square}, if $u_i \in S_r$ and $u_j \in S_s$ with $r<s$ and $(u_i:u_j)$ is not linear and is not divisible by any element of $N_H(x) \cup N_H(y)$, then the proof above shows there exists  $u_k \in S_p$ for some $p <s$ such that  $(u_k:u_j)$ is linear and divides $(u_i:u_j)$.
\end{remark}

We now extend \cref{l:add edge square} to all powers, establishing the edge addition reduction stated in Lemma \ref{Nursel_Susan}.

\begin{theorem}\label{t:add edge all powers}
Assume $H$, $x$, $y$, and $G$ are as in \cref{setting}. Assume further that $G$ is gapfree. Let $t \ge 2$ and assume that $I(H)^q$ has a linear quotients ordering for $2 \le q \le t$. Then $I(G)^t$ has linear quotients under \cref{ordering}.
\end{theorem}

\begin{proof}
We prove the result by induction on $t$. When $t=2$, the claim follows from \cref{l:add edge square}.

Let $S_1,\dots,S_{t+1}$ be the $S$-sets of $I(G)^t$ and let $\mathcal S_1,\dots,\mathcal S_t$ be the $S$-sets of $I(G)^{t-1}$. Let $u_j$ be fixed. We  prove that $J=((u_1,\dots ,u_{j-1}):u_j)$ is generated by variables by showing that for any $u_i < u_j$ there exists a $u_k < u_j$ with $(u_k:u_j)$ linear and dividing $(u_i:u_j)$. If $u_i, u_j \in S_r$ for some $r$, then the result follows from \cref{ordering within a set}. So, we may assume that $u_i \in S_r$ and $u_j \in S_p$ for some $r<p$ and $p \ge 2$. By \cref{l:neighbor sets} we also assume that $(u_i:u_j)$ has degree at least $2$ and is not divisible by any element of $N_H(x) \cup N_H(y)$. Thus it suffices to show the inductive step under these  assumptions. To do so, we prove the following stronger statement by induction on $t$.

\textbf{Claim.} Suppose $u_i\in S_r$ and $u_j\in S_p$ with $r<p$. If $(u_i:u_j)$ is not linear and is not divisible by any element of $N_H(x)\cup N_H(y)$, then there exists a generator $u_k\in S_s$, with $s<p$, such that $(u_k:u_j)$ is linear and $(u_k:u_j)\mid (u_i:u_j)$.

Note that when $t=2$, the base case remains true under the stronger assumptions of the claim by \cref{r:smaller set}. Now assume $t\ge 3$ and that the claim holds for $t-1$.

Since $u_i\in S_r$ and $u_j\in S_p$, choose optimal factorizations
\[
  u_i=f_1\cdots f_{t-r+1}(xy)^{r-1}  \text{ and }
u_j=e_1\cdots e_{t-p+1}(xy)^{p-1}
\]
where $f_1,\dots ,f_{t-r+1}\in E(H)$ and $ e_1,\dots ,e_{t-p+1}\in E(H) $.

First, suppose $r\geq 2$. By \Cref{l:reduce xy}, we have $\frac{u_i}{xy}\in \SS_{r-1}$ and $\frac{u_j}{xy}\in \SS_{p-1}$.  Moreover, $(\frac{u_i}{xy}: \frac{u_j}{xy})=(u_i:u_j)$. By the inductive hypothesis for $t-1$, there exists an element $\hat{u}\in\SS_{s}$ with $s<p-1$ such that $(\hat{u}:\frac{u_j}{xy})$ is a variable dividing $(\frac{u_i}{xy}:\frac{u_j}{xy})=(u_i:u_j)$.  By \Cref{l:reduce xy}, we have $\hat{u} xy \in S_{s'}$ for some $s'\le s+1< p$. Therefore, $\hat{u} xy < u_j$ and $(\hat{u} xy : u_j)= (\hat{u} :\frac{u_j}{xy} )$. Thus, $\hat{u} xy$ is the desired earlier generator.

Now assume $r=1$. We prove the claim in this case by a secondary induction on $d=\deg(u_i:u_j)\ge 2$. Since $r=1$, $u_i=f_1\cdots f_t \in S_1$. Thus we will show that for $u_j \in S_p$ with $p \ge 2$ and for every $u\in S_1$, if $\deg(u:u_j)\ge 2$ and $(u:u_j)$ is not divisible by any element of $N_H(x)\cup N_H(y)$, then there exists a generator $u_k\in S_s$ with $s<p$ such that $(u_k:u_j)$ is linear and $(u_k:u_j)\mid (u:u_j)$.

We first isolate a reduction that will be used several times below.

\noindent\textbf{Common edge reduction.}
Suppose that $u\in S_1$ and that an edge $e\in E(H)$ occurs both in a factorization of $u$ and in the chosen optimal factorization of $u_j$. Assume further that $(u:u_j)$ is not linear and that no variable in $N_H(x)\cup N_H(y)$ divides $(u:u_j)$. Then there exists a generator $u_k\in S_s$, with $s<p$, such that $(u_k:u_j)$ is linear and $(u_k:u_j)\mid (u:u_j)$. Indeed, by part~\textup{(2a)} of \Cref{l:reduce xy}, we have $\frac{u}{e}\in\mathcal S_1$ and  $\frac{u_j}{e}\in\mathcal S_p$. Moreover, $\left(\frac{u}{e}:\frac{u_j}{e}\right)=(u:u_j)$.  The induction hypothesis for $t-1$ therefore gives a generator $\widetilde u\in\mathcal S_s$, with $s<p$, such that $\left(\widetilde u:\frac{u_j}{e}\right)$ is linear and divides $(u:u_j)$. Set $u_k=e\widetilde u$. Since $e\neq xy$, part~\textup{(2b)} of \Cref{l:reduce xy} gives $u_k\in S_{s'}$ for some $s'\leq s<p$. Thus $u_k<u_j$ and $(u_k:u_j) = \left(\widetilde u:\frac{u_j}{e}\right)$. Hence $u_k$ has the required properties.

Applying the \textbf{common edge reduction} to $u=u_i$, we may assume from now on that
\[
\{f_1,\ldots,f_t\}
\cap
\{e_1,\ldots,e_{t-p+1}\}
=
\varnothing.
\]

Set $A=(u_i:u_j)$ and $B=(u_j:u_i)$. For each variable $w$, write $\alpha_w=\deg_w(A)$ and $\beta_w=\deg_w(B)$. Since
\[
\deg_w(u_i)-\deg_w(u_j)=\alpha_w-\beta_w
\]
for every $w$, and $u_i$ and $u_j$ have the same total degree, we have $\deg A=\deg B=d\geq 2$.

We use the following degree calculation repeatedly and we refer to it as $(\ast)$: Suppose that $f$ is an edge occurring in the chosen factorization of $u_i$, that $g$ is another edge of $H$, and that $u=\frac{u_i g}{f}$. Then, for every variable $w$,
\[
\deg_w(u:u_j)
=
\max\bigl\{
\alpha_w-\beta_w+\deg_w(g)-\deg_w(f),0
\bigr\}.
\]
In particular, if $\deg_w(g)-\deg_w(f)\leq \beta_w$ for every variable $w$,  then $\deg_w(u:u_j)\leq \alpha_w$ for every variable $w$. Therefore $(u:u_j)\mid A$.

Since $G$ is gapfree, each edge $f_k$ either has an endpoint in $\{x,y\}$ or is connected to the edge $\{x,y\}$. So, we can write $f_k= a_kb_k$ with $a_k\in N_H(x)\cup N_H(y)$.  Since no element of $N_H(x)\cup N_H(y)$ divides $A$, after relabeling the $f_k$'s, we may write $A=b_1\cdots b_d$ where  none of the $b_1,\ldots, b_d$ lies in $N_H(x)\cup N_H(y)$.

To complete the proof, we consider the following cases based on the form of $B$.

\noindent
\textbf{Case 1.} Suppose $B$ has a variable divisor outside $\{x,y\}$, say $v \in \operatorname{supp} (B)$ with $v\notin \{x,y\}$. Since $v$ occurs in $u_j$ outside the factors $xy$, after relabeling we may write $e_1=vz$. Since $v\mid B$, we have $\beta_v\geq1$.

First suppose that $z\nmid u_i$. Set $u=\frac{u_i e_1}{f_1}\in S_1$. We use $(\ast)$ with $g=e_1=vz$ and $f=f_1=a_1b_1$. Since $z\nmid u_i$ while $z\mid u_j$, we have $\beta_z=\deg_z(u_j)-\deg_z(u_i)\geq1$. For $w=v$ or $w=z$, we have
\[
\deg_w(e_1)-\deg_w(f_1)\leq1\leq\beta_w,
\]
and for every other variable $w$,
\[
\deg_w(e_1)-\deg_w(f_1)\leq0\leq\beta_w.
\]
Since the conditions in $(\ast)$ are met, we have $(u:u_j)\mid A$. If $(u:u_j)$ is linear, then $u$ is the desired earlier generator. Otherwise, since $(u:u_j)\mid A$, it is not divisible by any variable in $N_H(x)\cup N_H(y)$. Since $u$ and $u_j$ share the edge $e_1$ in their chosen factorizations, the \textbf{common edge reduction} therefore gives the desired earlier generator.

Now suppose that $z\mid u_i$. Choose an edge $f_q$ in the chosen factorization of $u_i$ that contains $z$ and set $u=\frac{u_i e_1}{f_q}\in S_1$. Since both $e_1$ and $f_q$ contain $z$, we have $\deg_z(e_1)-\deg_z(f_q)=0$. The only variable whose exponent can increase under this replacement from $u_i$ to $u$ is $v$. So, $\deg_v(e_1)-\deg_v(f_q)\leq1\leq\beta_v$. For all other variables the exponent change is nonpositive. Hence $(\ast)$ again gives $(u:u_j)\mid A$. If $(u:u_j)$ is linear, we are done. Otherwise, the \textbf{common edge reduction} applied to the common edge $e_1$ gives the desired earlier generator. This completes Case 1.

\noindent
\textbf{Case 2.}
Assume that $\operatorname{supp} (B) \subseteq \{x,y\}$ and $p\ge 3$. We choose an edge $f_q$ in the factorization of $u_i$ as follows.
\begin{itemize}
    \item If $x,y\in\operatorname{supp}(B)$, choose any $f_q$.
    \item If $B=x^d$, then $\beta_y=0$, so $\deg_y(u_i)\geq\deg_y(u_j)$.  Since the chosen optimal factorization of $u_j$ contains $(xy)^{p-1}$, we have $\deg_y(u_j)\geq p-1>0$. Thus $y\mid u_i$ and we choose $f_q$ containing $y$.
    \item  Similarly, if $B=y^d$, we choose $f_q$ containing $x$.
\end{itemize}

Set $u=\frac{u_i(xy)}{f_q}$. This factorization of $u$ uses one copy of $xy$. So $u\in S_{s'}$ for some  $s'\leq2$. Since $p\geq3$, we have $s'<p$. Hence $u<u_j$.

We next verify that $(u:u_j)\mid A$. If $x,y\in\operatorname{supp}(B)$, then $\beta_x,\beta_y\geq1$. For $w\notin\{x,y\}$, replacing the factor $f_q$ by $xy$ cannot increase the $w$ exponent since
\[
\deg_w(xy)-\deg_w(f_q)=-\deg_w(f_q)\leq0\leq\beta_w.
\]
For $w\in \{x,y\}$, the exponent can increase by at most one and $\deg_w(xy)-\deg_w(f_q)\leq1\leq\beta_w$. Thus
\[
\deg_w(xy)-\deg_w(f_q)\leq\beta_w
\]
for every variable $w$. Thus the conditions in $(\ast)$ hold and $(u:u_j)\mid A$.

If $B=x^d$, then $f_q$ contains $y$. So replacing $f_q$ by $xy$ does not increase the $y$ exponent. The only variable whose exponent can increase is $x$ and
\[
\deg_x(xy)-\deg_x(f_q)\leq1\leq\beta_x=d.
\]
Thus $(\ast)$ again holds. The case $B=y^d$ is symmetric. Therefore, in every case, $(u:u_j)\mid A$.

If $(u:u_j)$ is linear, then $u$ is the desired earlier generator. Suppose otherwise. Since $(u:u_j)\mid A$, it is not divisible by any variable in $N_H(x)\cup N_H(y)$. From the explicit factorization of $u$, we have $\frac{u}{xy}=\frac{u_i}{f_q} \in \SS_1$. Also, by \Cref{l:reduce xy}, we have $\frac{u_j}{xy}\in\SS_{p-1}$. In addition,  $\left(\frac{u}{xy}:\frac{u_j}{xy}\right)=(u:u_j)$. Since $1<p-1$, the induction hypothesis for $t-1$ gives a generator $\hat u\in\SS_s$ with $s<p-1$ such that $\left(\hat u:\frac{u_j}{xy}\right)$ is linear and divides $(u:u_j)$. By \Cref{l:reduce xy}, we have $\hat u(xy)\in S_{s''}$ for some $s''\leq s+1<p$.  Hence $\hat u(xy)<u_j$  and $(\hat u(xy):u_j) = \left(\hat u:\frac{u_j}{xy}\right)$. Thus $\hat u(xy)$ is the desired earlier generator. This completes Case 2.

\medskip

\noindent
\textbf{Case 3.}
Suppose that $p=2$ and $\operatorname{supp} (B) \subseteq \{x,y\}$.

First suppose that $x,y\in\operatorname{supp}(B)$. Recall that $f_1=a_1b_1$ such that $a_1\in N_H(x)\cup N_H(y)$ and $b_1\mid A$. Choose $z\in\{x,y\}$ such that $a_1z\in E(H)$  and set $u=\frac{u_i(a_1z)}{f_1} =\frac{u_i z}{b_1}$. Then $u\in S_1$ and hence $u<u_j$.

Since $b_1\mid A$ and $z\mid B$  while $\gcd(A,B)=1$, we have $b_1\neq z$. The replacement decreases the $b_1$ exponent by one, increases the $z$ exponent by one, and leaves every other exponent unchanged.  In other words, we have $\deg_{b_1}(u:u_j)=\alpha_{b_1}-1$ while $\deg_z(u:u_j) = \max\{-\beta_z+1,0\}$ because $\beta_z\geq1$. For every other variable $w$, we have $\deg_w(u:u_j)=\alpha_w$. Consequently, $(u:u_j)=\frac{A}{b_1}$. In particular,  $\deg(u:u_j)=d-1$.

If $d=2$, then $(u:u_j)$ is linear and $u$ is the desired generator. If $d\geq3$, then $2\leq\deg(u:u_j)=d-1<d$. Moreover, $(u:u_j)\mid A$, so it is not divisible by any variable in $N_H(x)\cup N_H(y)$. The secondary induction hypothesis applied to the pair $(u,u_j)$ therefore yields a generator $u_k\in S_s$ with $s<p$ such that $(u_k:u_j)$ is linear and divides $(u:u_j)$. Since $(u:u_j)\mid A$, this gives the desired generator.

The remaining cases to consider are $B=x^d$ and $B=y^d$. By symmetry, it is enough to treat $B=x^d$. In this case, $\deg_x(u_j)=\deg_x(u_i)+d$. Since $d\geq2$ and the chosen factorization of $u_j$ contains only one copy of $x$ coming from the factor $xy$, at least one of the $H$ edge factors of $u_j$ must contain $x$. After relabeling, write $e_1=xc$.

Since $\operatorname{supp}(B)=\{x\}$, we have $\beta_c=0$. As $c\mid u_j$, it follows that $\deg_c(u_i)\geq\deg_c(u_j)\geq1$. Choose an edge $f_q$ in the factorization of $u_i$ that contains $c$ and set $u=\frac{u_i e_1}{f_q}$. Then $u\in S_1$. So $u<u_j$. Also, notice that  $u$ and $u_j$ share the edge $e_1$.

Because both $e_1$ and $f_q$ contain $c$, the replacement does not increase the $c$ exponent. The only variable whose exponent can increase is $x$  and
\[
\deg_x(e_1)-\deg_x(f_q)\leq1\leq\beta_x=d.
\]
Thus the conditions in $(\ast)$ hold. So, we have $(u:u_j)\mid A$. If $(u:u_j)$ is linear, then $u$ is the desired generator. Otherwise, $(u:u_j)$ is not divisible by any variable in $N_H(x)\cup N_H(y)$, and the \textbf{common edge reduction} applied to $e_1$ gives the desired earlier generator. This completes the case $B=x^d$. The case $B=y^d$ follows by interchanging $x$ and $y$.
\end{proof}


\section{El Zahar and Erd\H{o}s reduction}\label{sec:domination}

After removing isolated vertices, a gapfree graph with at least one edge is connected. Throughout this section, $G$ is a connected gapfree graph with matching number $3$.  A set $D\subseteq V(G)$ is called a \emph{dominating set} of $G$ if every vertex outside $D$ is adjacent to at least one vertex in $D$. By \cite[Theorem~5]{EZE} (see also \cite{BT1990}), $G$ has a dominating set that induces either a clique or a path $P_3$ of length $2$.

We first treat domination by a path and then by a clique, considering the possible sizes of a minimal dominating clique. We combine these arguments with the removal of duplicate vertices, parallel vertices, and whiskers to reduce the problem to co-chordal graphs and graphs on at most seven vertices. The remaining finite list is enumerated in Appendix \ref{app:code} and displayed in Appendix \ref{app:graphs}.

\begin{lemma}\label{lem.P3Domination}
Suppose that $G$ has no duplicate vertices and is not dominated by a clique. Then $|V(G)|\leq 7$.
\end{lemma}

\begin{proof}
Suppose that $|V(G)|\geq 8$. By \cite[Theorem~5]{EZE}, $G$ has a dominating path $x-y-z$. Since no clique dominates $G$, we have $xz\notin E(G)$, and there are vertices $x_1,z_1$ such that $xx_1,zz_1\in E(G)$ and $yx_1,zx_1,xz_1,yz_1\notin E(G)$. Set
\begin{align*}
X&=N_G(x)\setminus\{y\}=\{x_1,\ldots,x_p\},\\
Y&=N_G(y)\setminus(X\cup\{x,z\})=\{y_1,\ldots,y_q\},\\
Z&=V(G)\setminus(\{x,y,z\}\cup X\cup Y)=\{z_1,\ldots,z_r\}.
\end{align*}
Every vertex of $Z$ is adjacent to $z$ but not to $x$ or $y$. We have $p,r\geq 1$, $q\geq 0$, and $p+q+r\geq 5$.

We first record three consequences of gapfreeness. The set $Z$ is independent since an edge inside $Z$ would form a gap with $xy$. Applying gapfreeness to $xx_i$ and $zz_j$ shows that every vertex of $X$ not adjacent to $z$ is adjacent to all of $Z$; in particular, $x_1z_j\in E(G)$ for every $j$. Applying gapfreeness to $yy_i$ and $x_1z_j$ shows that every vertex of $Y$ not adjacent to $x_1$ is adjacent to all of $Z$.

\medskip\noindent
\textbf{Case 1: $q=0$.}
No vertex $x_i$ is adjacent to both $z$ and a vertex $z_j\in Z$ since otherwise the triangle $x_i z z_j$ dominates $G$. Thus, with $A=X\cap N_G(z)$ and $B=X\setminus A$, there are no edges between $A$ and $Z$, while every vertex of $B$ is adjacent to all of $Z$. Consequently, $N_G(z_j)=\{z\}\cup B$ for every $j$. Since $G$ has no duplicate vertices, $r=1$ and $p\geq 4$.

There are no edges between $A$ and $B$ since an edge $ab$ with $a\in A$ and $b\in B$ would give the dominating triangle $xab$. Also, no $b\in B$ is adjacent to $y$ since $xyb$ would be a dominating triangle. Since $x_1\in B$, an edge from $A$ to $y$, or an edge inside $A$, would form a gap with $x_1z_1$. Finally, an edge inside $B$ would form a gap with $yz$. Therefore $X$ is independent and its vertices have only the two possible open neighborhoods $\{x,z\}$ and $\{x,z_1\}$. Since $p\geq 4$, two vertices of $X$ are duplicates, a contradiction.

\medskip\noindent
\textbf{Case 2: $q\geq 1$.}
We first show that $r=1$. Suppose that $r\geq 2$. If $p\geq 2$, then $\{xx_2,yy_1,zz_2,x_1z_1\}$ is a matching of size $4$, a contradiction. Thus $p=1$. If a vertex $y_i\in Y$ is adjacent to both $x_1$ and some $z_j$, then the triangle $x_1y_iz_j$ dominates $G$: every vertex of $Y$ is adjacent to $x_1$ or $z_j$, every vertex of $Z$ is adjacent to $x_1$, and $x,y,z$ are adjacent to $x_1,y_i,z_j$, respectively. Thus each $y_i\in Y$ is adjacent either to $x_1$ and no vertex of $Z$, or to all of $Z$ and not to $x_1$. All vertices of $Z$ therefore have the same open neighborhood, namely
\[
\{z,x_1\}\cup\{y_i\in Y:y_ix_1\notin E(G)\}.
\]
This gives duplicate vertices, a contradiction. Hence $r=1$ and $p+q\geq 4$.

Suppose that $p\geq 2$. We claim that there are no edges from $X$ to $Y\cup\{z\}$, and every vertex of $X$ is adjacent to $z_1$. If $p\geq 3$, an edge $x_i z$, necessarily with $i\neq 1$, gives the matching $\{x_i z,x_1z_1,xx_j,yy_1\}$ of size $4$ for $j\notin\{1,i\}$. Thus no vertex of $X$ is adjacent to $z$, so every vertex of $X$ is adjacent to $z_1$ by the second observation. An edge $x_i u$, with $u\in Y$, would then give the matching $\{x_i u,xx_j,x_kz_1,yz\}$ of size $4$ for distinct $j,k\neq i$.

If $p=2$, then $q\geq 2$. An edge $x_i u$ with $u\in Y$ gives the matching $\{x_i u,xx_h,yv,zz_1\}$ of size $4$ where $h\neq i$ and $v\in Y\setminus\{u\}$. Thus there are no $X$--$Y$ edges. The third observation now implies that every vertex of $Y$ is adjacent to $z_1$. An edge $x_2z$ would then give the matching $\{x_2z,xx_1,yy_1,y_2z_1\}$ of size $4$. Hence no vertex of $X$ is adjacent to $z$ and both vertices of $X$ are adjacent to $z_1$. This proves the claim.

There are also no edges from $X$ to $y$ since an edge $x_iy$ would give the dominating triangle $xyx_i$. Moreover, $X$ is independent since an edge inside $X$ would form a gap with $yz$. Hence $N_G(x_i)=\{x,z_1\}$ for every $i$, giving duplicate vertices, a contradiction. Therefore $p=1$ and $q\geq 3$.

The set $Y$ is independent since an edge $y_i y_j$ would give the matching $\{y_i y_j,yy_k,xx_1,zz_1\}$ of size $4$ for $k\notin\{i,j\}$. Each vertex of $Y$ is adjacent to at least one of $x_1,z_1$, but not both since $x_1y_i z_1$ would be a dominating triangle as above. Moreover, $y_i z\in E(G)$ forces $y_i x_1\in E(G)$ by gapfreeness applied to $y_i z$ and $xx_1$. Thus the only possible open neighborhoods of vertices of $Y$ are
\[
\{y,x_1\},\qquad \{y,x_1,z\},\qquad \{y,z_1\}.
\]
Since $q\geq 3$ and no vertices are duplicates, $q=3$ and all three neighborhoods occur. Choose distinct $u,v,w\in Y$ with $uz,vz_1\in E(G)$. Then $\{uz,vz_1,yw,xx_1\}$ is a matching of size $4$, a final contradiction.
\end{proof}

A dominating set $D$ of $G$ is \emph{minimal} if no proper subset of $D$ dominates $G$.

\begin{lemma}\label{lem.cliqueDomination}
Let $K_r$ be an $r$-clique whose vertex set is a minimal dominating set of $G$. Then $r\leq 3$.
\end{lemma}

\begin{proof}
The case $r=1$ is immediate. For $r\geq 2$, write $V(K_r)=\{x_1,\ldots,x_r\}$. By minimality, each $x_i$ has a vertex $y_i\notin V(K_r)$ with $N_G(y_i)\cap V(K_r)=\{x_i\}$. These vertices are distinct, so $\{x_iy_i:1\leq i\leq r\}$ is a matching. Hence $r\leq\nu(G)=3$.
\end{proof}

\begin{lemma}\label{lem.triangleDomination}
Suppose that $G$ has no duplicate vertices and no whiskers, and that $G$ has a minimal dominating triangle. Then $|V(G)|\leq 7$.
\end{lemma}

\begin{proof}
Suppose $|V(G)|\geq 8$. Let $D=\{x,y,z\}$ be a minimal dominating triangle. Choose $x$ to maximize $|N_G(x)\setminus D|$ among the vertices of $D$, and set $X=N_G(x)\setminus D$. Among the two remaining vertices of $D$, choose $y$ to maximize $|N_G(y)\setminus(D\cup X)|$. Set $Y=N_G(y)\setminus(D\cup X)$ and $Z=V(G)\setminus(D\cup X\cup Y)$. Write $X=\{x_1,\ldots,x_p\}$, $Y=\{y_1,\ldots,y_q\}$, and $Z=\{z_1,\ldots,z_r\}$. Every vertex of $Z$ is adjacent to $z$ but not to $x$ or $y$. By the choice of $x$, we have $p\geq q$. Every vertex of $Z$ is a neighbor of $z$ outside $D\cup X$, so the choice of $y$ gives $q\geq r$.

By minimality of $D$, we may choose the labels $x_1,y_1,z_1$ so that
\[
xx_1,yy_1,zz_1\in E(G)
\quad\text{and}\quad
yx_1,zx_1,xy_1,zy_1,xz_1,yz_1\notin E(G).
\]
Indeed, $\{y,z\}$ does not dominate $G$ but it dominates every vertex of $D$ because $D$ is a triangle. Thus some vertex outside $D$ is adjacent to neither $y$ nor $z$. Since $D$ dominates $G$, this vertex must be adjacent to $x$, and we may label it $x_1$. The same argument applied to $\{x,z\}$ and $\{x,y\}$ gives $y_1$ and $z_1$, respectively. By the definitions of the partition, $x_1\in X$, $y_1\in Y$, and $z_1\in Z$. So all three sets are nonempty. Consequently, $p\geq q\geq r\geq 1$. In addition, $|V(G)|\geq 8$ gives $p+q+r\geq 5$.

We first show that $r=1$. Notice that $Z$ is independent since an edge inside $Z$ would form a gap with $xy$. Suppose that $r\geq 2$. Then $p,q\geq 2$. Since $z_1$ is not a whisker and its only neighbor in $D$ is $z$, it must have a neighbor in $X\cup Y$. If $x_iz_1\in E(G)$, choose $k\neq i$ in $[p]$. Then $\{x_iz_1,xx_k,zz_2,yy_1\}$ is a matching of size $4$. Similarly, if $y_iz_1\in E(G)$, choose $k\neq i$ in $[q]$. Then $\{y_iz_1,yy_k,zz_2,xx_1\}$ is a matching of size $4$. Both possibilities contradict $\nu(G)=3$. Therefore $r=1$.

\medskip\noindent
\textbf{Case 1: $q\geq 2$.}
Then $p\geq 2$. There are no $X$--$Y$ edges: an edge $x_iy_j$ would give the matching $\{x_iy_j,xx_k,yy_\ell,zz_1\}$ of size $4$ where $k\neq i$ and $\ell\neq j$. Consequently, $Y$ is independent since an edge inside $Y$ would form a gap with $xx_1$.

Since $y_1$ is not a whisker and is adjacent to neither $x$ nor $z$, we have $y_1z_1\in E(G)$. Gapfreeness applied to $xx_1$ and $y_1z_1$ then gives $x_1z_1\in E(G)$ since $xy_1,xz_1,x_1y_1\notin E(G)$. If $y_2z\in E(G)$, then $\{y_2z,yy_1,xx_2,x_1z_1\}$ is a matching of size $4$, a contradiction. Thus both $y_1$ and $y_2$ have neighbors only in $\{y,z_1\}$. Since neither is a whisker, $N_G(y_1)=N_G(y_2)=\{y,z_1\}$, contradicting the absence of duplicate vertices.

\medskip\noindent
\textbf{Case 2: $q=1$.}
Then $r=1$ and $p\geq 3$. The set $X$ is independent since an edge $x_ix_j$ would give the matching $\{x_ix_j,xx_k,yy_1,zz_1\}$ of size $4$ where $k\notin\{i,j\}$.

Since $X$ is independent and $x_1y,x_1z\notin E(G)$, the only possible neighbors of $x_1$ other than $x$ are $y_1$ and $z_1$. As $x_1$ is not a whisker, it is adjacent to at least one of them. Interchanging $y,z$ and $y_1,z_1$ if necessary, assume $x_1y_1\in E(G)$. This interchange preserves the partition since $X$ is unchanged and the only vertices outside $D\cup X$ are $y_1,z_1$ with $zy_1,yz_1\notin E(G)$.

We next determine the possible neighbors of $x_2$ and $x_3$. For $i\in\{2,3\}$, let $j$ be the other index. Neither $x_iy$ nor $x_iz_1$ is an edge: these two possible edges would give the matchings $\{x_iy,x_1y_1,xx_j,zz_1\}$ and $\{x_iz_1,x_1y_1,xx_j,yz\}$ of size $4$, respectively. Since $X$ is independent, the only possible neighbors of $x_i$ other than $x$ are $y_1$ and $z$. As $x_i$ is not a whisker, it is adjacent to at least one of them. If $x_iz\in E(G)$, gapfreeness applied to $x_iz$ and $x_1y_1$ forces $x_iy_1\in E(G)$ since $x_ix_1,zx_1,zy_1\notin E(G)$. Thus both $x_2$ and $x_3$ are adjacent to $y_1$ and each has open neighborhood either $\{x,y_1\}$ or $\{x,y_1,z\}$.

Finally, consider the neighborhood of $x_1$. If $x_1z_1\in E(G)$, then $\{x_1z_1,x_2y_1,xx_3,yz\}$ is a matching of size $4$. Hence $x_1z_1\notin E(G)$ and $N_G(x_1)=\{x,y_1\}$. If either $x_2$ or $x_3$ is not adjacent to $z$, it has the same open neighborhood as $x_1$. Otherwise, $N_G(x_2)=N_G(x_3)=\{x,y_1,z\}$. In either case, two vertices are duplicates, a contradiction.

Both cases are impossible, so $|V(G)|\leq 7$.
\end{proof}

\begin{lemma}\label{lem:H-structure}
Let $H$ be a gapfree graph with $\nu(H)\leq 2$. Then either $H$ is co-chordal, or its unique nontrivial component is a pentagon and all other components are isolated vertices.
\end{lemma}

\begin{proof}
Gapfreeness allows at most one component containing an edge. Suppose $H$ is not co-chordal. Fr\"oberg's characterization \cite{F} and the regularity bound give $2<\reg~ I(H)\leq\nu(H)+1\leq 3$, so $\reg~ I(H)=3=\nu(H)+1$. By \cite[Theorem~11]{TNT}, the nontrivial component is a pentagon or a Cameron--Walker graph. The latter has induced matching number equal to its matching number, namely $2$, contradicting gapfreeness. Thus the nontrivial component is a pentagon.
\end{proof}

\begin{lemma}\label{lem:K2-reduction}
Suppose that $\{x,y\}$ is a minimal dominating set of $G$ and $xy\in E(G)$. Furthermore, suppose that $G$ has no duplicate or parallel vertices and no whiskers attached to $x$ or $y$. Then either $|V(G)|\leq 7$ or $G$ is co-chordal.
\end{lemma}

\begin{proof}
Suppose $|V(G)|\geq 8$. We show that $G$ is co-chordal. Set
\begin{align*}
X&=\{u\notin\{x,y\}:ux\in E(G),\ uy\notin E(G)\},\\
Y&=\{u\notin\{x,y\}:uy\in E(G),\ ux\notin E(G)\},\\
Z&=\{u\notin\{x,y\}:ux,uy\in E(G)\},
\end{align*}
and let $H=G[X\cup Y\cup Z]$. By minimality, $X,Y\neq\emptyset$: if either set were empty, one of $x,y$ would dominate $G$ alone. Write $p=|X|$, $q=|Y|$, and $r=|Z|$, interchanging $x$ and $y$ if necessary so that $p\geq q$. We have $p+q+r\geq 6$. The graph $H$ is gapfree and $\nu(H)\leq 2$ since a matching of size $3$ in $H$ could be extended by $xy$. Every vertex of $X\cup Y$ has a neighbor in $H$ since no whiskers are attached to $x$ or $y$.

We use the following observation throughout the proof.

\medskip\noindent
\textbf{Matching observation.} If $M$ is a matching of size $2$ in $H$, its uncovered vertices form an independent set contained entirely in $X$ or entirely in $Y$. Indeed, there are at least two uncovered vertices since $|V(H)|\geq 6$. They are independent since an edge between two of them would extend $M$ to a matching of size $3$ in $H$. If uncovered vertices $a\in X$ and $b\in Y$ exist, then $M\cup\{xa,yb\}$ is a matching of size $4$. If an uncovered vertex $z\in Z$ exists, choose another uncovered vertex $u$. Since $\{x,y\}$ dominates $G$, we can match $u$ to one of $x,y$. The vertex $z$ is adjacent to both, so it can be matched to the other. These two edges extend $M$ to a matching of size $4$. Both possibilities are impossible.

\medskip\noindent
\textbf{Case 1: $q,r\geq 2$.}
We first show that $X$ is independent. Suppose $x_1x_2\in E(G)$. An edge of $H$ from $y_i\in Y$ to a vertex outside $\{x_1,x_2\}$ would, together with $x_1x_2$, leave a vertex of $Z$ uncovered. Thus the matching observation and absence of whiskers give $\emptyset\neq N_H(y_i)\subseteq\{x_1,x_2\}$. Choose distinct $y_1,y_2\in Y$, relabeling $x_1,x_2$ so that $x_1y_1\in E(G)$. An edge $x_2y_2$ would give the matching $\{x_1y_1,x_2y_2,xz_1,yz_2\}$ of size $4$. Hence $N_H(y_2)=\{x_1\}$. Interchanging $y_1,y_2$ gives $N_H(y_1)=\{x_1\}$ as well. Since vertices of $Y$ are adjacent to $y$ but not $x$, we have $N_G(y_1)=N_G(y_2)=\{y,x_1\}$, giving duplicates, a contradiction. Thus $X$ is independent.

We next show that there are no $X$--$Y$ edges. Suppose $x_1y_1\in E(G)$ and choose $x_2\in X\setminus\{x_1\}$. If $x_2$ had a neighbor $u\neq y_1$ in $H$, the independence of $X$ would make $\{x_1y_1,x_2u\}$ a matching in $H$. It covers at most one vertex of $Z$, leaving another uncovered since $r\geq 2$. This contradicts the matching observation. Hence $N_H(x_2)=\{y_1\}$ since $x_2$ is not a whisker. Applying the same argument to the edge $x_2y_1$ and the vertex $x_1$ gives $N_H(x_1)=\{y_1\}$. Thus $N_G(x_1)=N_G(x_2)=\{x,y_1\}$, giving duplicates, a contradiction. With no $X$--$Y$ edges, an edge inside $Y$ would form a gap with $xx_1$, so $Y$ is independent as well.

Every vertex of $X\cup Y$ now has a neighbor in $Z$. Choose $x_1z_1\in E(G)$. For $y_i\in Y$, an edge $y_iz_j$ with $j\neq 1$ would give the matching $\{x_1z_1,y_iz_j,xx_2,yy_k\}$ of size $4$ where $x_2\neq x_1$ and $y_k\neq y_i$, a contradiction. Thus $N_G(y_i)=\{y,z_1\}$ for all $y_i\in Y$, again giving duplicates. Therefore this case is not possible.

\medskip\noindent
\textbf{Case 2: $q\geq 2$ and $r\leq 1$.}
Then $p\geq 3$. There is an $X$--$Y$ edge. Otherwise, any edge inside $X$ would form a gap with $yy_1$, so $X$ would be independent. Its vertices would then be whiskers if $r=0$, or would all have open neighborhood $\{x,z\}$ if $Z=\{z\}$, giving duplicate vertices.

Let $B$ be the graph on $X\cup Y$ consisting only of the $X$--$Y$ edges of $G$. Edges of $G$ within $X$ or within $Y$ are omitted, so $B$ is bipartite by construction. We claim that all edges of $B$ share an endpoint. If $x_1y_1,x_2y_2$ were disjoint edges of $B$, the matching observation would force $r=0$ and $q=2$ since an uncovered vertex $x_3\in X$ exists. Thus $p\geq 4$. Set $A=X\setminus\{x_1,x_2\}$. The matching observation shows that $A$ is independent. For $a\in A$, an edge $ax_1$ would give the matching $\{ax_1,x_2y_2,yy_1,xb\}$ of size $4$ for $b\in A\setminus\{a\}$; similarly, $ax_2\notin E(G)$. Hence each $N_H(a)$ is a nonempty subset of $\{y_1,y_2\}$. It cannot be the whole pair since then $a$ and $y$ would have the same open neighborhood $\{x,y_1,y_2\}$, a contradiction. Since vertices of $A$ are not duplicates, $A=\{a,b\}$ with $N_H(a)=\{y_1\}$ and $N_H(b)=\{y_2\}$. Now $x_1x_2\notin E(G)$ since $\{x_1x_2,ay_1,by_2,xy\}$ would be a matching of size $4$. It follows that $N_G(x_1)$ is either $\{x,y_1\}$ or $\{x,y_1,y_2\}$, making $x_1$ a duplicate of $a$ or $y$. This contradiction shows that $B$ has no two disjoint edges. Since $B$ is nonempty and bipartite, all its edges therefore share a common endpoint, proving the claim.

We now consider the location of this common endpoint. Suppose it is $x_1\in X$, and set $A=X\setminus\{x_1\}$. There are no edges from $A$ to $Y$. There are also none to $Z$: if $Z=\{z\}$ and $az\in E(G)$ for $a\in A$, choose $b\in A\setminus\{a\}$ and an edge $x_1y_i\in E(G)$. Then $\{az,x_1y_i,xb,yy_j\}$ is a matching of size $4$ for $j\neq i$. Thus every vertex of $A$ has its neighbors in $H$ contained in $A\cup\{x_1\}$. An edge inside $A$ would form a gap with $yy_1$, so $A$ is independent. Since its vertices are not whiskers, $N_G(a)=\{x,x_1\}$ for every $a\in A$. As $|A|=p-1\geq 2$, this gives duplicate vertices, a contradiction. If the common endpoint lies in $Y$ and $q\geq 3$, the same argument with $X$ and $Y$ interchanged applies: there are at least two other vertices in $Y$ and at least two vertices in $X$.

It remains that $Y=\{y_1,y_2\}$ and every $X$--$Y$ edge is incident to $y_1$. If $y_1y_2\in E(G)$, gapfreeness applied to $y_1y_2$ and $xx_i$ gives $x_iy_1\in E(G)$ for every $x_i\in X$. Moreover, $X$ is independent: if $ab\in E(G)$ with $a,b\in X$, choose $c\in X\setminus\{a,b\}$ and $d\in(X\cup Z)\setminus\{a,b,c\}$. These choices are possible because $p\geq 3$ and $p+r\geq 4$. Then $\{ab,cy_1,xd,yy_2\}$ is a matching of size $4$, a contradiction. Thus every vertex of $X$ has open neighborhood $\{x,y_1\}$, or $\{x,y_1,z\}$ when $Z=\{z\}$. As $p\geq 3$, two are duplicates, a contradiction.

Therefore $y_1y_2\notin E(G)$. Since $y_2$ is not a whisker, $Z=\{z\}$ and $y_2z\in E(G)$. Again $X$ is independent since an edge $ab$ inside $X$ would give the matching $\{ab,y_2z,yy_1,xc\}$ of size $4$ for $c\in X\setminus\{a,b\}$. Choose $a\in X$ with $ay_1\in E(G)$, which exists since there is an $X$--$Y$ edge. For $b\in X\setminus\{a\}$, an edge $bz$ would give the matching $\{ay_1,bz,xc,yy_2\}$ of size $4$ for $c\in X\setminus\{a,b\}$. Thus every vertex of $X\setminus\{a\}$ has open neighborhood $\{x,y_1\}$. There are at least two such vertices, a contradiction.

\medskip
\noindent
\textbf{Case 3: $q=1$.}
Write $Y=\{y_1\}$. We first show that $H$ is co-chordal. Otherwise, \Cref{lem:H-structure} gives a pentagon $C$ and isolated vertices. Since $|V(H)|\geq 6$, there is an isolated vertex $z$, which must lie in $Z$ because every vertex of $X\cup Y$ has a neighbor in $H$. Also, $y_1\in V(C)$ for the same reason. A perfect matching of $C-y_1$, together with $yy_1$ and $xz$, is a matching of size $4$, a contradiction. Thus $H$ is co-chordal.

Suppose that $G$ is not co-chordal. Let $C$ be an induced cycle of length at least $4$ in $\overline G$. The vertex $x$ is not on $C$ since its only neighbor in $\overline G$ is $y_1$. The vertex $y$ must therefore lie on $C$ since otherwise $C$ would lie in the chordal graph $\overline H$. Write $C=y,a,w_1,\ldots,w_s,b,y$. Since the neighbors of $y$ in $\overline G$ are exactly the vertices of $X$, we have $a,b\in X$. No $w_i$ lies in $X$ since then $yw_i$ would be a chord of $C$. Thus $w_1,\ldots,w_s\in Y\cup Z$. Consecutive pairs on $C$ are nonedges of $G$, while all nonconsecutive pairs are edges of $G$. If $s=1$, then $ab$ and $yw_1$ form a gap. So $s\geq 2$.

The edges $aw_s$ and $bw_1$ now form a matching $M$ of size $2$ in $H$. There are at least two uncovered vertices, and the matching observation places them all in $X$ since $|Y|=1$. Hence every vertex of $Y\cup Z$ is covered by $M$ and must be one of $w_1,w_s$. If $s\geq 3$, however, $w_2\in Y\cup Z$ would be uncovered. Thus $s=2$ and $Y\cup Z=\{w_1,w_2\}$. The remaining vertices of $H$ are exactly $A=X\setminus\{a,b\}$. By the matching observation, $A$ is independent, and $|A|=|V(H)|-4\geq 2$.

The five vertices of $C$ induce a pentagon $P$ in $G$, with cyclic order $y,w_1,b,a,w_2,y$. Choose distinct $c,d\in A$. Suppose that $c$ has a neighbor $v$ on $P$. Since $d\in X$, we have $xd\in E(G)$. A perfect matching of $P-v$, together with $cv$ and $xd$, is a matching of size $4$, a contradiction. Thus $c$ has no neighbor on $P$. It also has no neighbor in $A$ since $A$ is independent. As $V(G)=\{x\}\cup V(P)\cup A$ and $c\in X$, it follows that $N_G(c)=\{x\}$. Thus $c$ is a whisker attached to $x$, a contradiction. This proves that $G$ is co-chordal.
\end{proof}

\begin{lemma}\label{lem:K1-reduction}
Suppose that $x$ dominates $G$ and that no whiskers are attached to $x$. Then either $|V(G)|\leq 7$ or $G$ is co-chordal.
\end{lemma}

\begin{proof}
Suppose $|V(G)|\geq 8$ and set $H=G-x$. The graph $H$ is gapfree and has no isolated vertices since such a vertex would be a whisker attached to $x$ in $G$. Since $|V(H)|\geq 7$, a matching of size $3$ in $H$ would leave an uncovered vertex $u$. Adjoining $xu$ would then give a matching of size $4$ in $G$. Thus $\nu(H)\leq 2$. By \Cref{lem:H-structure}, if $H$ were not co-chordal, it would be a pentagon, contradicting $|V(H)|\geq 7$. Hence $H$ is co-chordal. Since $\overline G$ is obtained from $\overline H$ by adjoining the isolated vertex $x$, $G$ is co-chordal as well.
\end{proof}

\begin{lemma}\label{lem.26}
To prove \Cref{t:MAINtheorem}, it suffices to verify that $I(F)^s$ has linear quotients for every $s\geq 2$ for each of the $26$ graphs enumerated in Appendix \ref{app:code} and displayed in Appendix \ref{app:graphs}.
\end{lemma}

\begin{proof}
Suppose $G$ is a gapfree graph with $\nu(G)=3$. We prove this statement  by induction on $|V(G)|$.

If $G$ has a duplicate vertex, a parallel vertex, or a whisker, delete one vertex from the duplicate or parallel pair, or delete the degree one vertex, respectively, to obtain $H$. The graph $H$ is gapfree. Deleting one vertex decreases the matching number by at most one. So $\nu(H)\in\{2,3\}$. If $\nu(H)=3$, the induction hypothesis applies. If $\nu(H)=2$, \Cref{lem:H-structure} shows that $H$ is co-chordal or consists of a pentagon and isolated vertices. In either case, all $I(H)^s$ with $s\geq 2$ have linear quotients by \cite[Theorem~10.1.9]{HHgtm260} and \cite[Proposition~4.2]{EFHHKM}, respectively. Reintroducing the deleted vertex preserves this property by \Cref{duplication,expansion,t:add edge all powers}, respectively. Thus we may assume that $G$ is reduced, meaning that it has no duplicate or parallel vertices and no whiskers.

If no clique dominates $G$, \Cref{lem.P3Domination} gives $|V(G)|\leq 7$. Otherwise, choose a minimal dominating clique. Its cardinality is at most $3$ by \Cref{lem.cliqueDomination}. The cases of cardinalities $3$, $2$, and $1$ are handled by \Cref{lem.triangleDomination,lem:K2-reduction,lem:K1-reduction}, respectively. Hence either $|V(G)|\leq 7$ or $G$ is co-chordal. The co-chordal case is already known since all powers of its edge ideal have linear quotients \cite[Theorem~10.1.9]{HHgtm260}.

It remains to consider reduced, connected, gapfree, non-co-chordal graphs with matching number $3$ on at most $7$ vertices. The enumeration in Appendix \ref{app:code} gives exactly $26$ such graphs up to isomorphism. The verification of the statement for these graphs completes the induction.
\end{proof}


\section{Reducing the finite list of graphs}\label{sec:finite-reduction}

Let $G_1,\ldots,G_{26}$ denote the graphs in Appendix \ref{app:graphs}, numbered as in the enumeration in Appendix \ref{app:code}. We reduce the verification in \cref{lem.26} to $G_{12}$, shown in Figure \ref{fig:exceptional}.

These graphs are reduced with respect to the vertex deletions in the proof of \cref{lem.26}: they have no duplicate vertices, parallel vertices, or whiskers. In particular, none can be obtained from a smaller graph by duplication alone, since the last duplication would leave two vertices with the same open neighborhood. A construction may nevertheless use duplication at an intermediate stage: subsequent edge additions can distinguish the duplicate vertices and yield a reduced graph.

For the pentagon $C_5$, every power $I(C_5)^q$ with $q\geq 2$ has linear quotients by \cite[Proposition~4.2]{EFHHKM}. The same conclusion is known for $G_1$ and $G_2$ \cite[Section~5]{EFHHKM}. The vertex labels for $G_1$ and $G_2$ are those in Appendix \ref{app:graphs}; Figure \ref{fig:pentagon} fixes the labels for the pentagon.

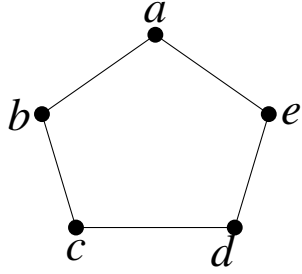
\begin{figure}
\centering
\begin{tikzpicture}[scale=1.5]
\coordinate (z) at (0,1){};
\coordinate (p) at (-1,0.3){};
\coordinate (q) at (1,0.3){};
\coordinate (a) at (-0.7,-0.7){};
\coordinate (b) at (0.7,-0.7){};
\fill(z)circle(0.7mm);
\fill(p)circle(0.7mm);
\fill(q)circle(0.7mm);
\fill(a)circle(0.7mm);
\fill(b)circle(0.7mm);
\draw(z)--(p)--(a)--(b)--(q)--cycle;
\draw(0,1.2)node{{\Large $a$}};
\draw(-1.2,0.3)node{{\Large $b$}};
\draw(1.2,0.3)node{{\Large $e$}};
\draw(-0.7,-0.9)node{{\Large $c$}};
\draw(0.6,-0.9)node{{\Large $d$}};
\end{tikzpicture}
\caption{The pentagon $C_5$.}\label{fig:pentagon}
\end{figure}

Table \ref{tab:graph-reductions} records the construction sequences for the remaining graphs. Each row lists the starting graph and then the operations in their prescribed order. Vertex labels refer to the starting graphs in Appendix \ref{app:graphs}, to the pentagon in Figure \ref{fig:pentagon}, and to the new vertices introduced in the row. The constructions are understood up to isomorphism, and the operations are performed in the stated order, with gapfreeness preserved at each edge addition, duplication,  and expansion. Thus Lemmas \ref{duplication} and \ref{expansion} and Theorem \ref{t:add edge all powers} transfer linear quotients for all powers $q\geq 2$ from the starting graph to the resulting graph.

\begingroup
\small
\renewcommand{\arraystretch}{1.2}
\begin{longtable}{@{}p{0.08\textwidth}p{0.12\textwidth}p{0.72\textwidth}@{}}
\caption{Constructions of the remaining graphs. Operations in each row are performed from left to right.}\label{tab:graph-reductions}\\
\hline
Graph & Start & Ordered operations\\
\hline
\endfirsthead
\hline
Graph & Start & Ordered operations\\
\hline
\endhead
\hline
\endfoot
$G_{3}$ & $G_1$ & Add $\{b,e\}$.\\
$G_{4}$ & $G_2$ & Add $\{a,x\}$.\\
$G_{5}$ & $C_5$ & Duplicate $e$ to create $x$ and $c$ to create $y$; add $\{c,x\}$.\\
$G_{6}$ & $G_1$ & Duplicate $e$ to create $y$; add $\{e,x\}$.\\
$G_{7}$ & $C_5$ & Duplicate $b$ to create $x$ and $d$ to create $y$; add $\{d,x\}$, then $\{b,e\}$.\\
$G_{8}$ & $G_2$ & Duplicate $a$ to create $y$; add $\{a,x\}$.\\
$G_{9}$ & $C_5$ & Duplicate $c$ to create $x$ and $e$ to create $y$; add $\{e,x\}$, $\{a,c\}$, $\{c,x\}$, in this order.\\
$G_{10}$ & $G_1$ & Duplicate $b$ to create $y$; add $\{b,e\}$, then $\{b,x\}$.\\
$G_{11}$ & $G_1$ & Duplicate $d$ to create $y$; add $\{b,y\}$.\\
$G_{13}$ & $G_1$ & Expand $x$ to create $y$; add $\{e,x\}$.\\
$G_{14}$ & $G_2$ & Duplicate $e$ to create $y$; add $\{b,e\}$.\\
$G_{15}$ & $G_2$ & Duplicate $a$ to create $y$; add $\{x,y\}$, then $\{a,d\}$.\\
$G_{16}$ & $G_2$ & Duplicate $e$ to create $y$; add $\{c,y\}$.\\
$G_{17}$ & $C_5$ & Duplicate $b$ twice to create $x$ and $y$; add $\{e,x\}$, $\{d,y\}$, $\{b,d\}$, $\{x,y\}$, in this order.\\
$G_{18}$ & $G_{12}$ & Add $\{b,y\}$.\\
$G_{19}$ & $G_1$ & Expand $x$ to create $y$; add $\{b,x\}$, then $\{e,x\}$.\\
$G_{20}$ & $G_1$ & Expand $x$ to create $y$; add $\{b,x\}$, then $\{e,y\}$.\\
$G_{21}$ & $C_5$ & Duplicate $c$ to create $x$ and $e$ to create $y$; add $\{x,y\}$, $\{e,x\}$, $\{c,y\}$, $\{c,x\}$, $\{e,y\}$, in this order.\\
$G_{22}$ & $G_2$ & Duplicate $x$ to create $y$; add $\{a,x\}$.\\
$G_{23}$ & $G_1$ & Duplicate $c$ to create $y$; add $\{e,y\}$, $\{b,e\}$, $\{c,y\}$, in this order.\\
$G_{24}$ & $G_2$ & Duplicate $a$ to create $y$; add $\{c,y\}$, $\{a,d\}$, $\{a,y\}$, in this order.\\
$G_{25}$ & $G_2$ & Add $\{a,x\}$; duplicate $d$ to create $y$; add $\{a,d\}$, then $\{d,y\}$.\\
$G_{26}$ & $G_{12}$ & Add $\{b,e\}$, $\{a,x\}$, $\{x,y\}$.\\
\end{longtable}
\endgroup
These reductions establish the following lemma.

\begin{lemma}
\label{pending}
In order to prove Theorem \ref{t:MAINtheorem}, it suffices to verify the statement for the graph $G_{12}$ in Figure \ref{fig:exceptional}.
\end{lemma}


\section{Proof of Theorem \ref{t:MAINtheorem}}\label{sec:main-proof}
By \cref{pending}, it remains to show that every power $I(G)^s$ with $s\geq 2$ has linear quotients for the graph $G=G_{12}$ in Figure \ref{fig:exceptional}.

For the ordering below, we write $f=y$ and $g=x$ for the two additional vertices in Figure \ref{fig:exceptional}. The pentagon vertices retain the labels $a,b,c,d,e$.

\begin{notation}\label{not:fig1_order}
Let $G$ be the graph in Figure \ref{fig:exceptional}, relabeled as above, on the vertex set $ V(G)=\{a,b,c,d,e,f,g\} $ with edge set
\begin{align*}
e_1&=ab, & e_2&=ae, & e_3&=af, & e_4&=bc,\\
e_5&=bg, & e_6&=cd, & e_7&=cg, & e_8&=cf,\\
e_9&=de, & e_{10}&=dg, & e_{11}&=df, & e_{12}&=eg.
\end{align*}
Let $I=I(G)$. We order the generators of $I^2$ as follows:
\begin{align*}
M^{(2)}:\quad&
a^2b^2<a^2be<a^2bf<ab^2c<ab^2g<abcf\\
&<a^2e^2<a^2ef<abce<abcd<abcg<abeg\\
&<abde<acef<adef<ade^2<acde<adeg\\
&<a^2f^2<abfg<abdg<abdf<acf^2<acdf\\
&<acfg<adf^2<adfg<aefg<ae^2g<aceg\\
&<b^2c^2<b^2cg<bc^2d<bc^2g<bc^2f<bcdg\\
&<bcde<bcdf<bceg<b^2g^2<bcg^2<bcfg\\
&<bdg^2<bdeg<bdfg<beg^2<c^2df<cd^2f\\
&<cd^2e<c^2d^2<cd^2g<c^2fg<cdfg<cdeg\\
&<c^2dg<cdg^2<c^2f^2<cdef<cdf^2<cefg\\
&<ceg^2<c^2g^2<d^2eg<d^2e^2<d^2ef<de^2g\\
&<d^2fg<d^2g^2<deg^2<d^2f^2<defg<e^2g^2 .
\end{align*}
For $s\geq 3$, define an order $M^{(s)}$ on $\mathcal G(I^s)$ recursively by
\[
M^{(s)}
=
\operatorname{reduced}
\big(
M^{(s-1)}e_1<
M^{(s-1)}e_2<
\cdots<
M^{(s-1)}e_{12}
\big),
\]
where $\operatorname{reduced}$ means that a generator is omitted if it has already appeared in an earlier edge block.

For $s\geq 3$, let $A_i^{(s)}$ denote the $e_i$-block of $M^{(s)}$. Thus, $w\in A_i^{(s)}$ means that $w=e_i \cdot u$ for some $u\in \mathcal G(I^{s-1})$ and that $w$ does not appear in any earlier $e_j$-block with $j<i$. We also set
\[
A_{<i}^{(s)}=\bigcup_{h<i} A_h^{(s)}.
\]

For $w\in \mathcal G(I^s)$, define
\[
\operatorname{set}_s(w)
=
\{x\in V(G): \text{ there exists } w'<w \text{ in } M^{(s)}
\text{ with } w':w=x\}.
\]
Let $w\in A_i^{(s)}$. Define a set of variables $D_i(w)\subseteq \operatorname{set}_s(w)$ such that for every $x\in D_i(w)$, there exists some $w_x\in A_{<i}^{(s)}$ with $w_x:w=x$.

\end{notation}

The first lemma transfers variables in the colon ideals from one power to the next. The following two lemmas handle the comparisons between different blocks, including the exceptional support sets.

\begin{lemma}\label{lem:block_inheritance}
Let $s\geq 3$, and suppose $w=e_i \cdot  u\in A_i^{(s)}$. If $x\in \operatorname{set}_{s-1}(u)$, then $x\in \operatorname{set}_s(w)$.
\end{lemma}

\begin{proof}
Choose $u'<u$ in $M^{(s-1)}$ such that $u':u=x$. Then $e_i \cdot  u'$ appears before $e_i \cdot u$ inside the $e_i$-block, unless $e_i \cdot u'$ has already appeared in an earlier edge block. In either case, $e_i \cdot  u'$ occurs before $w=e_i \cdot  u$ in $M^{(s)}$ and $(e_i \cdot u'):(e_i u)=u':u=x$. Thus $x\in \operatorname{set}_s(w)$.
\end{proof}

\begin{lemma}\label{lem:terminal_support_sets}
For every $s\geq 2$ and every $u\in \mathcal G(I^s)$, the following implications hold:
\[
\begin{aligned}
&\operatorname{supp}(u)\subseteq \{c,d,g\}
\quad\Longrightarrow\quad
e,f\in \operatorname{set}_s(u),\\
&\operatorname{supp}(u)\subseteq \{d,e,f\}
\quad\Longrightarrow\quad
g\in \operatorname{set}_s(u),\\
&\operatorname{supp}(u)\subseteq \{d,g\}
\quad\Longrightarrow\quad
f\in \operatorname{set}_s(u).
\end{aligned}
\]
\end{lemma}

\begin{proof}
For $s=2$, this follows directly from the order $M^{(2)}$. For the generators supported on $\{c,d,g\}$, the following earlier witnesses occur in $M^{(2)}$:
\[
\begin{array}{c|cc}
u & \text{$e$-witness} & \text{$f$-witness}\\
\hline
c^2d^2 & cd^2e & c^2df\\
cd^2g & cd^2e & cd^2f\\
c^2dg & cdeg & c^2df\\
cdg^2 & cdeg & cdfg\\
c^2g^2 & ceg^2 & c^2fg\\
d^2g^2 & d^2eg & d^2fg .
\end{array}
\]
Similarly,
\[
d^2eg:d^2e^2=g,\qquad
d^2eg:d^2ef=g,\qquad
d^2fg:d^2f^2=g,
\]
and $d^2fg:d^2g^2=f$. Thus the three implications hold for $s=2$.

Now assume $s\geq 3$. We prove the first implication; the other two are identical. Suppose $\operatorname{supp}(u)\subseteq \{c,d,g\}$. Since the only edges of $G$ supported on $\{c,d,g\}$ are $cd,cg,dg$, the first block in which $u$ appears is one of the $cd$-, $cg$- or $dg$-blocks. Thus $u=e_i \cdot  u'$ with $e_i\in \{cd,cg,dg\}$ and $\operatorname{supp}(u')\subseteq \{c,d,g\}$. By induction on $s$, we have $e,f\in \operatorname{set}_{s-1}(u')$. Then Lemma \ref{lem:block_inheritance} implies $e,f\in \operatorname{set}_s(u)$.

The proofs of the second and third implications are the same, using the edge sets $\{de,df\}$ and $\{dg\}$, respectively.
\end{proof}

\begin{lemma}\label{lem:block_check_compact}
Let $s\geq 3$, let $w\in A_i^{(s)}$, and let $v\in A_{<i}^{(s)}$. For each $i$, the following statements hold.

\begin{enumerate}
    \item If $i=1$, then there is no such $v$.

    \item We have
\begin{align*}
 D_2(w)&=
\begin{cases}
\{b\},& b\nmid w,\\
\{a,b,d,f,g\},& b\mid w
\end{cases},\\
    D_3(w)&=
\begin{cases}
\{b,e\},&
g\nmid w\text{ or }(b\nmid w\text{ and }e\nmid w),\\
\{a,b,c,d,e\},&
g\mid w\text{ and }(b\mid w\text{ or }e\mid w)
\end{cases},\\
    D_4(w) &=\{a\},\\ D_5(w) & =
\begin{cases}
\{a,c\},& c\nmid w,\\
\{a,b,c,d,e\},& c\mid w
\end{cases},\\
    D_8(w)&=
\begin{cases}
\{a,b,d,g\},& d\nmid w,\ e\nmid w,\ g\nmid w,\\
\{a,b,c,d,g\},& d\mid w\text{ or }e\mid w\text{ or }g\mid w
\end{cases},\\
    D_{11}(w)&=
\begin{cases}
\{a,c,e,g\},& e\nmid w\text{ and }g\nmid w,\\
\{a,b,c,d,e,g\},& e\mid w\text{ or }g\mid w
\end{cases},\\
    D_{12}(w)&=\{a,b,c,d\},
    \end{align*}
    and $v:w\in (D_i(w))$ for $i=2,3,4,5,8,11, 12$.

    \item We have
\begin{align*}
D_6(w)&=
\begin{cases}
\{b\},& e\nmid w\text{ and }f\nmid w,\\
\{a,b\},& e\mid w\text{ or }f\mid w,
\end{cases}\\
    D_7(w)&= \{b,d\} \cup
\begin{cases}
\{a\},& e\mid w\text{ or }f\mid w,\\
\varnothing,& e\nmid w\text{ and }f\nmid w
\end{cases}
    \cup
\begin{cases}
\{c,e\},& d\mid w,\\
\varnothing,& d\nmid w.
\end{cases}
    \end{align*}
    Moreover, for $i=6,7$, either $v:w\in (D_i(w))$ or
\[
\operatorname{supp}(w)\subseteq \{c,d,g\}
\quad\text{and}\quad
e\mid(v:w)\text{ or }f\mid(v:w).
\]

    \item If $i=9$, then
\[
D_9(w)=
\begin{cases}
\{a,c\},& g\nmid w,\\
\{a,b,c\},& g\mid w
\end{cases}
 .
\]
    Moreover, either $v:w\in (D_9(w))$ or
\[
\operatorname{supp}(w)\subseteq \{d,e,f\}
\quad\text{and}\quad
g\mid(v:w).
\]

    \item If $i=10$, then
\[
D_{10}(w)=
\{b,c,e\}
\cup
\begin{cases}
\{a\},& e\mid w\text{ or }f\mid w,\\
\varnothing,& e\nmid w\text{ and }f\nmid w
\end{cases}
\cup
\begin{cases}
\{d\},& e\mid w,\\
\varnothing,& e\nmid w
\end{cases}
 .
\]
    Moreover, either $v:w\in (D_{10}(w))$ or
\[
\operatorname{supp}(w)\subseteq \{d,g\}
\quad\text{and}\quad
f\mid(v:w).
\]
\end{enumerate}
\end{lemma}

\begin{proof}
The case $i=1$ is immediate. We use the phrase ``$v$ comes from the $xy$-block'' to mean that $v\in A_h^{(s)}$ with $e_h=xy$.

We first consider the cases in (2). For $i=2$, write $w=ae\cdot m$ where $m\in \mathcal G(I^{s-1})$. Since the $ab$-block precedes the $ae$-block, the generator $ab\cdot m$ occurs before $w$, and $(ab\cdot m):w=b$. Hence $b\in\operatorname{set}_s(w)$. If $b\mid w$, then the $b$-factor cannot come from $ab$, because $w\notin A_1^{(s)}$, and it cannot come from $bg$ because $ae\cdot bg=ab\cdot eg$ would again put $w$ in the $ab$-block. Thus $w=ae\cdot bc\cdot m'$. The earlier generators $ab\cdot ae\cdot m'$, $ab\cdot de\cdot m'$, $ab\cdot cf\cdot m'$, and $ab\cdot eg\cdot m'$ give quotients $a,d,f,g$, respectively. Therefore the displayed formula for $D_2(w)$ follows.

Now let $v\in A_{<2}^{(s)}$, so $v=ab\cdot n$. If $b\nmid w$, then $b\mid(v:w)$. Hence $v:w\in(D_2(w))$. Suppose $b\mid w$, so $w=ae\cdot bc\cdot m'$ and $D_2(w)=\{a,b,d,f,g\}$. Since $w\notin A_1^{(s)}$, we have $(ce) m'\notin I^{s-1}$. Thus $m'$ contains no edge entirely inside $D_2(w)$; otherwise, using $(ce) ab=ae\cdot bc$, $(ce) af=ae\cdot cf$, $(ce) bg=bc\cdot eg$, $(ce) dg=cd\cdot eg$ or $(ce) df=cf\cdot de$, we would get $(ce) m'\in I^{s-1}$, a contradiction. Hence every edge appearing in $m'$ has exactly one endpoint in $D_2(w)$. Since $D_2(w)$ is a vertex cover of $G$, every edge appearing in $n$ has at least one endpoint in $D_2(w)$. Comparing the number of $D_2(w)$-variables in $n$ and in $(ce) m'$, and using $(ab\cdot n):(ab\cdot (ce) m')=n:((ce) m')$, one variable from $D_2(w)$ remains in the quotient. Thus $v:w\in(D_2(w))$.

The cases $i=3,4,5,8,11,12$ are checked in the same way. The basic witnesses are as follows: for $i=3$, the $ab$- and $ae$-blocks give $b$ and $e$; for $i=4$, the $ab$-block gives $a$; for $i=5$, the $ab$- and $bc$-blocks give $a$ and $c$; for $i=8$, the $af$-, $bc$-, $cd$-, and $cg$-blocks give $a,b,d,g$; for $i=11$, the $af$-, $cf$-, $de$-, and $dg$-blocks give $a,c,e,g$; and for $i=12$, the $ae$-, $bg$-, $cg$-, and $de$-blocks give $a,b,c,d$. Whenever one of the exceptional variables in the displayed branch of $D_i(w)$ divides $w$, the block condition forces that variable to occur through a later edge, and replacing that later edge by the corresponding earlier edge gives the additional variables listed in $D_i(w)$. Reading through the earlier blocks then shows that every $v:w$ is divisible by one of the variables in the displayed $D_i(w)$. This proves (2).

We now prove (3). For $i=6$, write $w=cd\cdot m$. Since every edge containing $a$ or $b$ occurs before the $cd$-block, we have $a,b\nmid w$. The generator $bc\cdot m$ occurs before $w$, and $(bc\cdot m):w=b$, so $b\in\operatorname{set}_s(w)$. If $e\mid w$ or $f\mid w$, then replacing the edge containing $e$ by $ae$ or the edge containing $f$ by $af$, gives an earlier generator whose quotient by $w$ is $a$. Hence the displayed formula for $D_6(w)$ follows.

For the colon statement when $i=6$, if $e\mid w$ or $f\mid w$, then $D_6(w)=\{a,b\}$. Every earlier block is one of $ab,ae,af,bc,bg$. So every $v\in A_{<6}^{(s)}$ is divisible by $a$ or $b$. Since $a,b\nmid w$, we get $v:w\in(a,b)=(D_6(w))$. If $e,f\nmid w$, then $D_6(w)=\{b\}$. Since $a,b,e,f\nmid w$, we have $\operatorname{supp}(w)\subseteq\{c,d,g\}$. If $v:w\notin(b)$, then $v$ cannot come from an earlier block whose edge contains $b$. The only remaining earlier blocks are the $ae$- and $af$-blocks, and these give $e\mid(v:w)$ or $f\mid(v:w)$, respectively.

For $i=7$, write $w=cg\cdot m$. Again $a,b\nmid w$. The $bg$- and $cd$-blocks give quotients $b$ and $d$, so $b,d\in\operatorname{set}_s(w)$. If $e\mid w$ or $f\mid w$, replacing the edge containing $e$ by $ae$ or the edge containing $f$ by $af$, gives the quotient $a$. If $d\mid w$, then the $d$-factor cannot come from $cd$, and it also cannot come from $de$, since $cg\cdot de=cd\cdot eg$ would put $w$ in the $cd$-block. Thus the $d$-factor comes from $dg$ or $df$. The earlier $cd$-block replacements $cg\cdot dg\mapsto cd\cdot cg$ and $cg\cdot df\mapsto cd\cdot cf$ give quotient $c$, while $cg\cdot dg\mapsto cd\cdot eg$ and $cg\cdot df\mapsto cd\cdot eg$ give quotient $e$. This gives the displayed formula for $D_7(w)$.

The colon statement for $i=7$ is then read block-by-block. Quotients from the $ab$-, $bc$-, and $bg$-blocks are divisible by $b$. Quotients from the $cd$-block are divisible by $d$, unless $d\mid w$, in which case they are accounted for by $c$ or $e$. Quotients from the $ae$- and $af$-blocks are divisible by $e$ or $f$, unless $e\mid w$ or $f\mid w$, in which case they are accounted for by $a$. Therefore, if $v:w\notin(D_7(w))$, then $e\mid(v:w)$ or $f\mid(v:w)$. In this exceptional situation $e,f\nmid w$, and since also $a,b\nmid w$, we have $\operatorname{supp}(w)\subseteq\{c,d,g\}$. This proves (3).

We prove (4). Suppose $w\in A_9^{(s)}$, and write $w=de\cdot m$. Since every edge containing $a$, $b$ or $c$ occurs before the $de$-block, we have $a,b,c\nmid w$. The earlier generators $ae\cdot m$ and $cd\cdot m$ give quotients $a$ and $c$, respectively. If $g\mid w$, replacing the edge containing $g$ by $bg$ gives an earlier generator whose quotient by $w$ is $b$. Hence the displayed formula for $D_9(w)$ follows.

If $g\mid w$, then $D_9(w)=\{a,b,c\}$. Every earlier block contains one of $a,b,c$. Since none of these variables divides $w$, we have $v:w\in(D_9(w))$. If $g\nmid w$, then $D_9(w)=\{a,c\}$, and $a,b,c,g\nmid w$, so $\operatorname{supp}(w)\subseteq\{d,e,f\}$. If $v:w\notin(a,c)$, then $v$ cannot come from any earlier block except the $bg$-block. Therefore $g\mid(v:w)$, proving the asserted alternative for $i=9$.

Finally we prove (5). Suppose $w\in A_{10}^{(s)}$, and write $w=dg\cdot m$. Since every edge containing $a$, $b$ or $c$ occurs before the $dg$-block, we have $a,b,c\nmid w$. The earlier generators $bg\cdot m$, $cg\cdot m$, and $de\cdot m$ give quotients $b,c,e$, respectively. So $b,c,e\in\operatorname{set}_s(w)$. If $e\mid w$, then the $e$-factor must come from an $eg$-factor; otherwise $w$ would occur in the $ae$- or $de$-block. The quotients $ae\cdot dg:dg\cdot eg=a$ and $de\cdot dg:dg\cdot eg=d$ show that $a,d\in\operatorname{set}_s(w)$. If $f\mid w$, then the $f$-factor must come from a $df$-factor; otherwise $w$ would occur in the $af$- or $cf$-block. The quotient $af\cdot dg:dg\cdot df=a$ shows that $a\in\operatorname{set}_s(w)$. This gives the displayed formula for $D_{10}(w)$.

Now let $v\in A_{<10}^{(s)}$. If $e\nmid w$ and $f\nmid w$, then $D_{10}(w)=\{b,c,e\}$, and since $a,b,c,e,f\nmid w$, we have $\operatorname{supp}(w)\subseteq\{d,g\}$. Quotients from the $ab$-, $bc$-, and $bg$-blocks are divisible by $b$; quotients from the $cd$-, $cg$-, and $cf$-blocks are divisible by $c$; and quotients from the $ae$- and $de$-blocks are divisible by $e$. Thus, if $v:w\notin(D_{10}(w))$, the only remaining possibility is that $v$ comes from the $af$-block, and then $f\mid(v:w)$.

If $e\nmid w$ but $f\mid w$, then $D_{10}(w)=\{a,b,c,e\}$. The same block reading shows that the $ab$-, $bc$-, and $bg$-blocks are caught by $b$, the $cd$-, $cg$-, and $cf$-blocks are caught by $c$, the $ae$- and $de$-blocks are caught by $e$, and the $af$-block is caught by $a$. Hence $v:w\in(D_{10}(w))$.

It remains to consider $e\mid w$ where $D_{10}(w)=\{a,b,c,d,e\}$. The $ab$-, $bc$-, and $bg$-blocks are caught by $b$, the $cd$-, $cg$-, and $cf$-blocks are caught by $c$, and the $ae$- and $af$-blocks are caught by $a$. For the $de$-block, suppose toward a contradiction that neither $d$ nor $e$ divides $v:w$. Since $v:w\notin(D_{10}(w))$ also rules out $a,b,c$, the monomial $v$ is supported on $\{d,e,f,g\}$. We may write $w=(dg)^p(df)^r(eg)^q$ and $v=(de)^t(dg)^P(df)^R(eg)^Q$ where $p,q,t\geq 1$ and $p+r+q=t+P+R+Q=s$. The assumptions $d\nmid(v:w)$ and $e\nmid(v:w)$ give $t+P+R\leq p+r$ and $t+Q\leq q$. The second inequality gives $P+R=s-t-Q\geq s-q=p+r$, while the first gives $P+R\leq p+r-t<p+r$, a contradiction. Thus $d\mid(v:w)$ or $e\mid(v:w)$, so $v:w\in(D_{10}(w))$.

All cases have now been covered.
\end{proof}

We now combine the preceding lemmas to verify linear quotients for the recursively defined orderings.

\begin{lemma}\label{prop:fig1_all_powers_lq}
Let $G$ be the graph in Figure \ref{fig:exceptional} and let $I=I(G)$. Then $I^s$ has linear quotients for all $s\geq 2$.
\end{lemma}

\begin{proof}
We prove that $M^{(s)}$ is a linear quotients ordering for $I^s$ for every $s\geq 2$. The case $s=2$ can be verified by a finite check. Now let $s\geq 3$, and assume that $M^{(s-1)}$ is a linear quotients ordering for $I^{s-1}$.

It is enough to show that whenever $v<w$ are generators in $M^{(s)}$ with $\deg(v:w)>1$, the monomial $v:w$ is divisible by a variable in $\operatorname{set}_s(w)$. Write $w\in A_i^{(s)}$ and $v\in A_j^{(s)}$. Since $v<w$, either $j=i$ or $j<i$.

First suppose $j=i$. Then $v=e_i\cdot v'$ and $w=e_i\cdot w'$ where $v'<w'$ in $M^{(s-1)}$ and $v:w=v':w'$. Since $\deg(v':w')>1$, the induction hypothesis gives a variable $x\in\operatorname{set}_{s-1}(w')$ such that $x\mid(v':w')$. By Lemma \ref{lem:block_inheritance}, this same variable belongs to $\operatorname{set}_s(w)$. Therefore $v:w$ is divisible by a variable in $\operatorname{set}_s(w)$.

Now suppose $j<i$. Let $v\in A_{<i}^{(s)}$. By Lemma \ref{lem:block_check_compact}, either $v:w\in(D_i(w))$ or one of the exceptional alternatives for $i=6,7,9,10$ occurs. If $v:w\in(D_i(w))$, then $v:w$ is divisible by some variable in $D_i(w)$. This is enough because $D_i(w)\subseteq \operatorname{set}_s(w)$.

It remains to handle the exceptional alternatives. If $i=6$ or $i=7$, then $\operatorname{supp}(w)\subseteq\{c,d,g\}$ and $e\mid(v:w)$ or $f\mid(v:w)$. Since $w\in A_i^{(s)}$, we write $w=e_i\cdot w'$ where $e_i=cd$ if $i=6$ and $e_i=cg$ if $i=7$. Then $\operatorname{supp}(w')\subseteq\{c,d,g\}$, so Lemma \ref{lem:terminal_support_sets} gives $e,f\in\operatorname{set}_{s-1}(w')$. Then $e,f\in\operatorname{set}_s(w)$ by Lemma \ref{lem:block_inheritance}. Since $v:w$ is divisible by $e$ or by $f$, the desired conclusion follows.

If $i=9$, then $\operatorname{supp}(w)\subseteq\{d,e,f\}$ and $g\mid(v:w)$. Write $w=de\cdot w'$. Then $\operatorname{supp}(w')\subseteq\{d,e,f\}$. It follows from Lemma \ref{lem:terminal_support_sets} that $g\in\operatorname{set}_{s-1}(w')$. Thus $g\in\operatorname{set}_s(w)$ by Lemma \ref{lem:block_inheritance}. Since $g\mid(v:w)$, the desired conclusion follows.

Finally, if $i=10$, then $\operatorname{supp}(w)\subseteq\{d,g\}$ and $f\mid(v:w)$. Write $w=dg\cdot w'$. Then $\operatorname{supp}(w')\subseteq\{d,g\}$. It follows from  Lemma \ref{lem:terminal_support_sets} that  $f\in\operatorname{set}_{s-1}(w')$. Thus, by Lemma \ref{lem:block_inheritance}, we have $f\in\operatorname{set}_s(w)$. Since $f\mid(v:w)$, the desired conclusion follows. This completes the proof.
\end{proof}


\section{Further directions}\label{sec:future}

The next case following this combinatorial approach is matching number four.

\begin{question}\label{q:matching-four}
If $G$ is gapfree with $\nu(G)=4$, does $I(G)^s$ have a linear resolution for every $s\geq2$? Can one prove linear quotients for all these powers?
\end{question}

More generally, one can ask how a bound on matching number determines a uniform starting power. In \cite{NP}, Nevo-Peeva considered the complement of the icosahedral graph which has matching number $6$ and regularity of its edge ideal is $4$. It also has a nonlinear square and a linear cube. Thus a uniform bound beginning at the cube for matching number at most $6$, if true, would have the earliest possible starting power.

\begin{question}\label{q:uniform-start}
For every gapfree graph $G$ with at least one edge, do the following implications hold?
\begin{enumerate}[(a)]
    \item If $\nu(G)\leq5$, then $I(G)^s$ has a linear resolution for every $s\geq2$.
    \item If $\nu(G)\leq6$, then $I(G)^s$ has a linear resolution for every $s\geq3$.
\end{enumerate}
\end{question}

Our result suggests a broader question.

\begin{question}
What natural combinatorial conditions on a gapfree graph $G$ ensure that $I(G)^s$ has linear quotients for every $s\geq2$?
\end{question}




\section*{Acknowledgments}

This work began during the authors' stay at the American Institute of Mathematics (AIM) as part of the AIM SQuaRE program. The authors thank AIM for its warm hospitality and for providing a stimulating research environment. N.~Erey gratefully acknowledges funding and support from the Fulbright Visiting Scholar Program, under which part of this work was carried out during a research visit to the University of Minnesota. T.~H\`a was partially supported by a Simons Foundation grant. S.~Kara was supported by NSF Grant DMS-2418805.

\appendix
\section{Enumeration and graph catalogue}\label{app:enumeration}

This appendix provides the enumeration used in Lemma~\ref{lem.26}
and the complete list of $26$ graphs used in
Section~\ref{sec:finite-reduction}. We denote these graphs by
$G_1,\ldots,G_{26}$: the subscript is the Python output number.

\subsection{Python enumeration}\label{app:code}

The following Python code gives the $26$ graphs used in Lemma~\ref{lem.26}. It starts with the list of graphs on at most seven vertices provided by NetworkX. We keep the gapfree graphs, apply the vertex deletions from the proof of the lemma, and retain the resulting non-co-chordal graphs with matching number $3$. Each graph is counted only once up to isomorphism. The catalogue uses the same graph numbers, with letter labels on the vertices.

\begingroup
\footnotesize
\begin{verbatim}
import itertools
import os
import networkx as nx
import matplotlib.pyplot as plt
from matplotlib.backends.backend_pdf import PdfPages

def is_gap_free(G):
    """
    G is gapfree iff G has no induced 2K2.
    Equivalently, complement(G) has no induced C4.
    """
    H = nx.complement(G)

    for S in itertools.combinations(H.nodes(), 4):
        sub = H.subgraph(S)

        if sub.number_of_edges() == 4:
            if nx.is_connected(sub) and all(d == 2 for _, d in sub.degree()):
                return False

    return True

def matching_number(G):
    M = nx.max_weight_matching(G, maxcardinality=True)
    return len(M)

def edge_ideal_has_linear_resolution(G):
    """
    By Froberg's theorem, I(G) has a linear resolution
    iff complement(G) is chordal.
    """
    return nx.is_chordal(nx.complement(G))

def twin_vertices_to_remove(G):
    """
    Vertices x and y are twins if

        N(x) \\ {y} = N(y) \\ {x}.

    This includes both adjacent and non-adjacent twins.

    For each twin class, keep one vertex and remove the rest.
    """
    remove = set()
    nodes = sorted(G.nodes())

    for i, x in enumerate(nodes):
        if x in remove:
            continue

        Nx = set(G.neighbors(x))

        for y in nodes[i + 1:]:
            if y in remove:
                continue

            Ny = set(G.neighbors(y))

            if Nx - {y} == Ny - {x}:
                remove.add(y)

    return remove

def reduce_graph(G):
    """
    Repeatedly perform the operations in this order:

    (1) remove leaves;
    (2) remove one of the two twin vertices;
    (3) remove isolated vertices.

    Continue until the graph no longer changes.
    """
    H = nx.convert_node_labels_to_integers(G.copy())

    changed = True

    while changed:
        changed = False

        # (1) remove leaves
        leaves = {v for v in H.nodes() if H.degree(v) == 1}
        if leaves:
            H.remove_nodes_from(leaves)
            H = nx.convert_node_labels_to_integers(H)
            changed = True

        # (2) remove one vertex from each twin pair
        twins = twin_vertices_to_remove(H)
        if twins:
            H.remove_nodes_from(twins)
            H = nx.convert_node_labels_to_integers(H)
            changed = True

        # (3) remove isolated vertices
        isolated = {v for v in H.nodes() if H.degree(v) == 0}
        if isolated:
            H.remove_nodes_from(isolated)
            H = nx.convert_node_labels_to_integers(H)
            changed = True

    return nx.convert_node_labels_to_integers(H)

def all_gap_free_graphs_at_most_7_vertices():
    """
    NetworkX graph_atlas_g contains one representative of each
    unlabeled simple graph on at most 7 vertices.
    """
    for G in nx.graph_atlas_g():
        if G.number_of_nodes() <= 7:
            G = nx.convert_node_labels_to_integers(G)

            if is_gap_free(G):
                yield G

def add_if_new(graphs, G):
    for H in graphs:
        if nx.is_isomorphic(G, H):
            return False

    graphs.append(G.copy())
    return True

def main():
    final_graphs = []

    total_gap_free = 0

    for G in all_gap_free_graphs_at_most_7_vertices():
        total_gap_free += 1

        H = reduce_graph(G)

        if H.number_of_nodes() == 0:
            continue

        if matching_number(H) != 3:
            continue

        # (4) remove graphs whose edge ideals have linear resolutions
        if edge_ideal_has_linear_resolution(H):
            continue

        add_if_new(final_graphs, H)

    print(f"Number of gapfree graphs checked: {total_gap_free}")
    print(f"Final number of non-isomorphic graphs: {len(final_graphs)}")

    output_pdf = r"graphs.pdf"
    print("Saving to:", os.path.abspath(output_pdf))

    if len(final_graphs) == 0:
        print("No graphs found — nothing to save.")
        return

    with PdfPages(output_pdf) as pdf:
        for i, G in enumerate(final_graphs, start=1):
            plt.figure(figsize=(5, 5))

            pos = nx.spring_layout(G, seed=100 + i)

            nx.draw_networkx_nodes(G, pos, node_size=500)
            nx.draw_networkx_edges(G, pos, width=1.5)
            nx.draw_networkx_labels(G, pos, font_size=10)

            plt.title(
                f"Graph {i}: |V|={G.number_of_nodes()}, "
                f"|E|={G.number_of_edges()}, matching number=3"
            )

            plt.axis("off")
            plt.tight_layout()
            pdf.savefig()
            plt.close()

if __name__ == "__main__":
    main()    
\end{verbatim}
\endgroup

\subsection{The catalogue of the twenty six graphs}\label{app:graphs}
This catalogue uses the Python numbering.
For ease of comparison, the vertices $a,b,c,d,e$ are placed around
the outside of each drawing, with the additional vertices $x$ and,
when present, $y$ inside. Vertex labels specify the displayed
representative of each isomorphism class.
The graphs $G_1$ and $G_2$ are the known starting cases, and
$G_{12}$ is the graph treated in Section~\ref{sec:main-proof}.
The numbers of vertices and edges are given below each drawing.

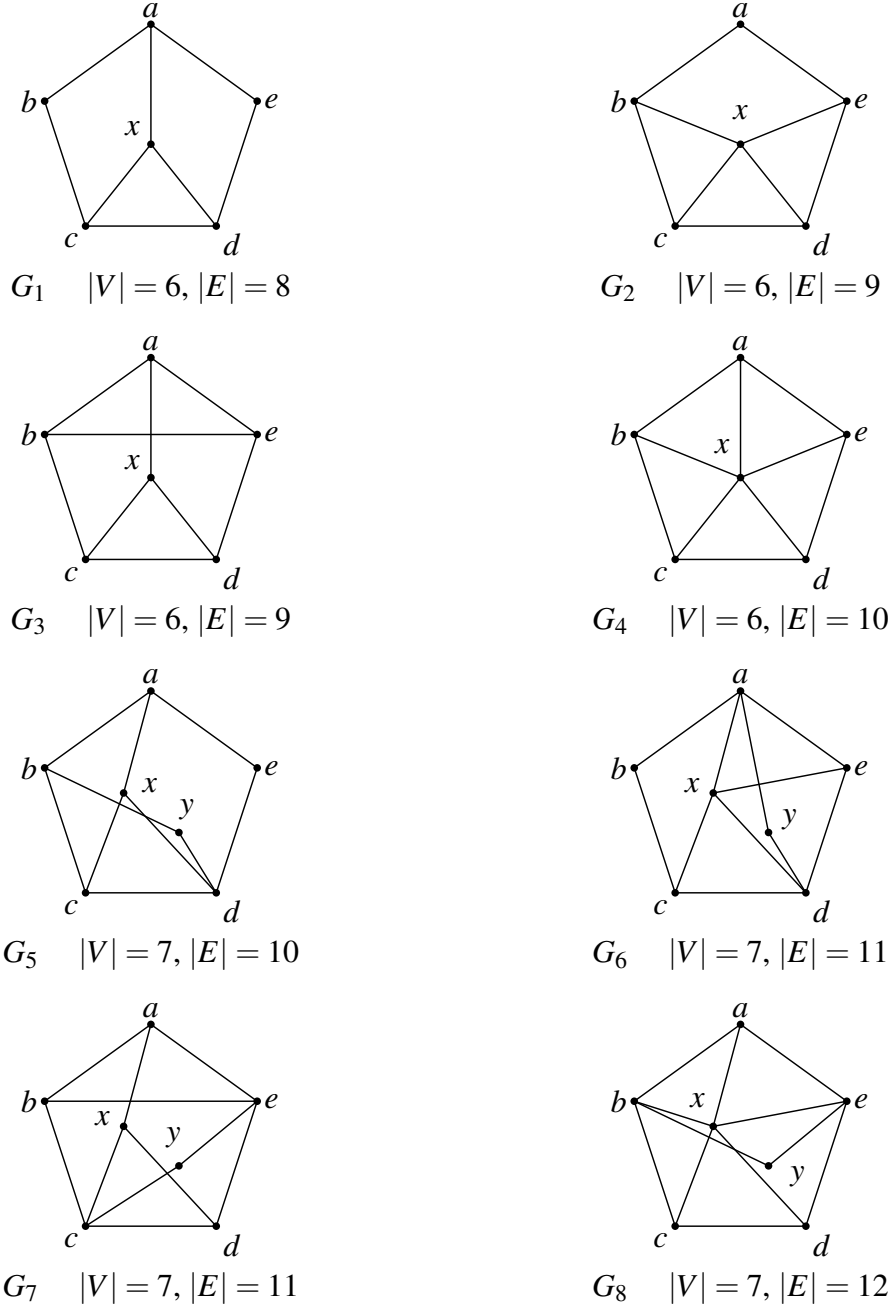
\begin{figure}[!ht]
\centering
\begin{minipage}{0.48\textwidth}
\centering
\begin{tikzpicture}[scale=0.95, line width=0.55pt]
\path[use as bounding box] (-1.85,-1.62) rectangle (1.85,1.95);
\coordinate (a) at (0,1.55);
\coordinate (b) at (-1.48,0.48);
\coordinate (c) at (-0.91,-1.26);
\coordinate (d) at (0.91,-1.26);
\coordinate (e) at (1.48,0.48);
\coordinate (x) at (0,-0.12);
\draw (a)--(b);
\draw (a)--(e);
\draw (a)--(x);
\draw (b)--(c);
\draw (c)--(d);
\draw (c)--(x);
\draw (d)--(e);
\draw (d)--(x);
\fill (a) circle (1.5pt) node[above, inner sep=2.5pt] {$a$};
\fill (b) circle (1.5pt) node[left, inner sep=2.5pt] {$b$};
\fill (c) circle (1.5pt) node[below left, inner sep=2.5pt] {$c$};
\fill (d) circle (1.5pt) node[below right, inner sep=2.5pt] {$d$};
\fill (e) circle (1.5pt) node[right, inner sep=2.5pt] {$e$};
\fill (x) circle (1.5pt);
\node[inner sep=0pt] at (-0.2563,0.0997) {$x$};
\end{tikzpicture}
\par\smallskip
$G_{1}$ \quad $|V|=6$, $|E|=8$
\end{minipage}
\hfill
\begin{minipage}{0.48\textwidth}
\centering
\begin{tikzpicture}[scale=0.95, line width=0.55pt]
\path[use as bounding box] (-1.85,-1.62) rectangle (1.85,1.95);
\coordinate (a) at (0,1.55);
\coordinate (b) at (-1.48,0.48);
\coordinate (c) at (-0.91,-1.26);
\coordinate (d) at (0.91,-1.26);
\coordinate (e) at (1.48,0.48);
\coordinate (x) at (0,-0.12);
\draw (a)--(b);
\draw (a)--(e);
\draw (b)--(c);
\draw (b)--(x);
\draw (c)--(d);
\draw (c)--(x);
\draw (d)--(e);
\draw (d)--(x);
\draw (e)--(x);
\fill (a) circle (1.5pt) node[above, inner sep=2.5pt] {$a$};
\fill (b) circle (1.5pt) node[left, inner sep=2.5pt] {$b$};
\fill (c) circle (1.5pt) node[below left, inner sep=2.5pt] {$c$};
\fill (d) circle (1.5pt) node[below right, inner sep=2.5pt] {$d$};
\fill (e) circle (1.5pt) node[right, inner sep=2.5pt] {$e$};
\fill (x) circle (1.5pt);
\node[inner sep=0pt] at (0.0000,0.3193) {$x$};
\end{tikzpicture}
\par\smallskip
$G_{2}$ \quad $|V|=6$, $|E|=9$
\end{minipage}
\par\vspace{3mm}
\begin{minipage}{0.48\textwidth}
\centering
\begin{tikzpicture}[scale=0.95, line width=0.55pt]
\path[use as bounding box] (-1.85,-1.62) rectangle (1.85,1.95);
\coordinate (a) at (0,1.55);
\coordinate (b) at (-1.48,0.48);
\coordinate (c) at (-0.91,-1.26);
\coordinate (d) at (0.91,-1.26);
\coordinate (e) at (1.48,0.48);
\coordinate (x) at (0,-0.12);
\draw (a)--(b);
\draw (a)--(e);
\draw (a)--(x);
\draw (b)--(c);
\draw (b)--(e);
\draw (c)--(d);
\draw (c)--(x);
\draw (d)--(e);
\draw (d)--(x);
\fill (a) circle (1.5pt) node[above, inner sep=2.5pt] {$a$};
\fill (b) circle (1.5pt) node[left, inner sep=2.5pt] {$b$};
\fill (c) circle (1.5pt) node[below left, inner sep=2.5pt] {$c$};
\fill (d) circle (1.5pt) node[below right, inner sep=2.5pt] {$d$};
\fill (e) circle (1.5pt) node[right, inner sep=2.5pt] {$e$};
\fill (x) circle (1.5pt);
\node[inner sep=0pt] at (-0.2563,0.0997) {$x$};
\end{tikzpicture}
\par\smallskip
$G_{3}$ \quad $|V|=6$, $|E|=9$
\end{minipage}
\hfill
\begin{minipage}{0.48\textwidth}
\centering
\begin{tikzpicture}[scale=0.95, line width=0.55pt]
\path[use as bounding box] (-1.85,-1.62) rectangle (1.85,1.95);
\coordinate (a) at (0,1.55);
\coordinate (b) at (-1.48,0.48);
\coordinate (c) at (-0.91,-1.26);
\coordinate (d) at (0.91,-1.26);
\coordinate (e) at (1.48,0.48);
\coordinate (x) at (0,-0.12);
\draw (a)--(b);
\draw (a)--(e);
\draw (a)--(x);
\draw (b)--(c);
\draw (b)--(x);
\draw (c)--(d);
\draw (c)--(x);
\draw (d)--(e);
\draw (d)--(x);
\draw (e)--(x);
\fill (a) circle (1.5pt) node[above, inner sep=2.5pt] {$a$};
\fill (b) circle (1.5pt) node[left, inner sep=2.5pt] {$b$};
\fill (c) circle (1.5pt) node[below left, inner sep=2.5pt] {$c$};
\fill (d) circle (1.5pt) node[below right, inner sep=2.5pt] {$d$};
\fill (e) circle (1.5pt) node[right, inner sep=2.5pt] {$e$};
\fill (x) circle (1.5pt);
\node[inner sep=0pt] at (-0.2563,0.3193) {$x$};
\end{tikzpicture}
\par\smallskip
$G_{4}$ \quad $|V|=6$, $|E|=10$
\end{minipage}
\par\vspace{3mm}
\begin{minipage}{0.48\textwidth}
\centering
\begin{tikzpicture}[scale=0.95, line width=0.55pt]
\path[use as bounding box] (-1.85,-1.62) rectangle (1.85,1.95);
\coordinate (a) at (0,1.55);
\coordinate (b) at (-1.48,0.48);
\coordinate (c) at (-0.91,-1.26);
\coordinate (d) at (0.91,-1.26);
\coordinate (e) at (1.48,0.48);
\coordinate (x) at (-0.38,0.13);
\coordinate (y) at (0.39,-0.42);
\draw (a)--(b);
\draw (a)--(e);
\draw (a)--(x);
\draw (b)--(c);
\draw (b)--(y);
\draw (c)--(d);
\draw (c)--(x);
\draw (d)--(e);
\draw (d)--(x);
\draw (d)--(y);
\fill (a) circle (1.5pt) node[above, inner sep=2.5pt] {$a$};
\fill (b) circle (1.5pt) node[left, inner sep=2.5pt] {$b$};
\fill (c) circle (1.5pt) node[below left, inner sep=2.5pt] {$c$};
\fill (d) circle (1.5pt) node[below right, inner sep=2.5pt] {$d$};
\fill (e) circle (1.5pt) node[right, inner sep=2.5pt] {$e$};
\fill (x) circle (1.5pt);
\node[inner sep=0pt] at (-0.0139,0.2398) {$x$};
\fill (y) circle (1.5pt);
\node[inner sep=0pt] at (0.4998,-0.0905) {$y$};
\end{tikzpicture}
\par\smallskip
$G_{5}$ \quad $|V|=7$, $|E|=10$
\end{minipage}
\hfill
\begin{minipage}{0.48\textwidth}
\centering
\begin{tikzpicture}[scale=0.95, line width=0.55pt]
\path[use as bounding box] (-1.85,-1.62) rectangle (1.85,1.95);
\coordinate (a) at (0,1.55);
\coordinate (b) at (-1.48,0.48);
\coordinate (c) at (-0.91,-1.26);
\coordinate (d) at (0.91,-1.26);
\coordinate (e) at (1.48,0.48);
\coordinate (x) at (-0.38,0.13);
\coordinate (y) at (0.39,-0.42);
\draw (a)--(b);
\draw (a)--(e);
\draw (a)--(x);
\draw (a)--(y);
\draw (b)--(c);
\draw (c)--(d);
\draw (c)--(x);
\draw (d)--(e);
\draw (d)--(x);
\draw (d)--(y);
\draw (e)--(x);
\fill (a) circle (1.5pt) node[above, inner sep=2.5pt] {$a$};
\fill (b) circle (1.5pt) node[left, inner sep=2.5pt] {$b$};
\fill (c) circle (1.5pt) node[below left, inner sep=2.5pt] {$c$};
\fill (d) circle (1.5pt) node[below right, inner sep=2.5pt] {$d$};
\fill (e) circle (1.5pt) node[right, inner sep=2.5pt] {$e$};
\fill (x) circle (1.5pt);
\node[inner sep=0pt] at (-0.6729,0.2398) {$x$};
\fill (y) circle (1.5pt);
\node[inner sep=0pt] at (0.6829,-0.2369) {$y$};
\end{tikzpicture}
\par\smallskip
$G_{6}$ \quad $|V|=7$, $|E|=11$
\end{minipage}
\par\vspace{3mm}
\begin{minipage}{0.48\textwidth}
\centering
\begin{tikzpicture}[scale=0.95, line width=0.55pt]
\path[use as bounding box] (-1.85,-1.62) rectangle (1.85,1.95);
\coordinate (a) at (0,1.55);
\coordinate (b) at (-1.48,0.48);
\coordinate (c) at (-0.91,-1.26);
\coordinate (d) at (0.91,-1.26);
\coordinate (e) at (1.48,0.48);
\coordinate (x) at (-0.38,0.13);
\coordinate (y) at (0.39,-0.42);
\draw (a)--(b);
\draw (a)--(e);
\draw (a)--(x);
\draw (b)--(c);
\draw (b)--(e);
\draw (c)--(d);
\draw (c)--(x);
\draw (c)--(y);
\draw (d)--(e);
\draw (d)--(x);
\draw (e)--(y);
\fill (a) circle (1.5pt) node[above, inner sep=2.5pt] {$a$};
\fill (b) circle (1.5pt) node[left, inner sep=2.5pt] {$b$};
\fill (c) circle (1.5pt) node[below left, inner sep=2.5pt] {$c$};
\fill (d) circle (1.5pt) node[below right, inner sep=2.5pt] {$d$};
\fill (e) circle (1.5pt) node[right, inner sep=2.5pt] {$e$};
\fill (x) circle (1.5pt);
\node[inner sep=0pt] at (-0.6729,0.2398) {$x$};
\fill (y) circle (1.5pt);
\node[inner sep=0pt] at (0.3168,-0.0173) {$y$};
\end{tikzpicture}
\par\smallskip
$G_{7}$ \quad $|V|=7$, $|E|=11$
\end{minipage}
\hfill
\begin{minipage}{0.48\textwidth}
\centering
\begin{tikzpicture}[scale=0.95, line width=0.55pt]
\path[use as bounding box] (-1.85,-1.62) rectangle (1.85,1.95);
\coordinate (a) at (0,1.55);
\coordinate (b) at (-1.48,0.48);
\coordinate (c) at (-0.91,-1.26);
\coordinate (d) at (0.91,-1.26);
\coordinate (e) at (1.48,0.48);
\coordinate (x) at (-0.38,0.13);
\coordinate (y) at (0.39,-0.42);
\draw (a)--(b);
\draw (a)--(e);
\draw (a)--(x);
\draw (b)--(c);
\draw (b)--(x);
\draw (b)--(y);
\draw (c)--(d);
\draw (c)--(x);
\draw (d)--(e);
\draw (d)--(x);
\draw (e)--(x);
\draw (e)--(y);
\fill (a) circle (1.5pt) node[above, inner sep=2.5pt] {$a$};
\fill (b) circle (1.5pt) node[left, inner sep=2.5pt] {$b$};
\fill (c) circle (1.5pt) node[below left, inner sep=2.5pt] {$c$};
\fill (d) circle (1.5pt) node[below right, inner sep=2.5pt] {$d$};
\fill (e) circle (1.5pt) node[right, inner sep=2.5pt] {$e$};
\fill (x) circle (1.5pt);
\node[inner sep=0pt] at (-0.5997,0.4961) {$x$};
\fill (y) circle (1.5pt);
\node[inner sep=0pt] at (0.7927,-0.5664) {$y$};
\end{tikzpicture}
\par\smallskip
$G_{8}$ \quad $|V|=7$, $|E|=12$
\end{minipage}
\caption{Graphs $G_{1}$--$G_{8}$ in the Python enumeration.}
\end{figure}
\clearpage

\begin{figure}[p]
\centering
\begin{minipage}{0.48\textwidth}
\centering
\begin{tikzpicture}[scale=0.95, line width=0.55pt]
\path[use as bounding box] (-1.85,-1.62) rectangle (1.85,1.95);
\coordinate (a) at (0,1.55);
\coordinate (b) at (-1.48,0.48);
\coordinate (c) at (-0.91,-1.26);
\coordinate (d) at (0.91,-1.26);
\coordinate (e) at (1.48,0.48);
\coordinate (x) at (-0.12,0.13);
\coordinate (y) at (0.39,-0.42);
\draw (a)--(b);
\draw (a)--(c);
\draw (a)--(e);
\draw (a)--(y);
\draw (b)--(c);
\draw (b)--(x);
\draw (c)--(d);
\draw (c)--(x);
\draw (d)--(e);
\draw (d)--(x);
\draw (d)--(y);
\draw (e)--(x);
\fill (a) circle (1.5pt) node[above, inner sep=2.5pt] {$a$};
\fill (b) circle (1.5pt) node[left, inner sep=2.5pt] {$b$};
\fill (c) circle (1.5pt) node[below left, inner sep=2.5pt] {$c$};
\fill (d) circle (1.5pt) node[below right, inner sep=2.5pt] {$d$};
\fill (e) circle (1.5pt) node[right, inner sep=2.5pt] {$e$};
\fill (x) circle (1.5pt);
\node[inner sep=0pt] at (-0.06,0.51) {$x$};
\fill (y) circle (1.5pt);
\node[inner sep=0pt] at (0.6829,-0.2369) {$y$};
\end{tikzpicture}
\par\smallskip
$G_{9}$ \quad $|V|=7$, $|E|=12$
\end{minipage}
\hfill
\begin{minipage}{0.48\textwidth}
\centering
\begin{tikzpicture}[scale=0.95, line width=0.55pt]
\path[use as bounding box] (-1.85,-1.62) rectangle (1.85,1.95);
\coordinate (a) at (0,1.55);
\coordinate (b) at (-1.48,0.48);
\coordinate (c) at (-0.91,-1.26);
\coordinate (d) at (0.91,-1.26);
\coordinate (e) at (1.48,0.48);
\coordinate (x) at (-0.38,0.13);
\coordinate (y) at (0.39,-0.42);
\draw (a)--(b);
\draw (a)--(e);
\draw (a)--(x);
\draw (a)--(y);
\draw (b)--(c);
\draw (b)--(e);
\draw (b)--(x);
\draw (c)--(d);
\draw (c)--(x);
\draw (c)--(y);
\draw (d)--(e);
\draw (d)--(x);
\fill (a) circle (1.5pt) node[above, inner sep=2.5pt] {$a$};
\fill (b) circle (1.5pt) node[left, inner sep=2.5pt] {$b$};
\fill (c) circle (1.5pt) node[below left, inner sep=2.5pt] {$c$};
\fill (d) circle (1.5pt) node[below right, inner sep=2.5pt] {$d$};
\fill (e) circle (1.5pt) node[right, inner sep=2.5pt] {$e$};
\fill (x) circle (1.5pt);
\node[inner sep=0pt] at (-0.8193,-0.0531) {$x$};
\fill (y) circle (1.5pt);
\node[inner sep=0pt] at (0.6829,-0.4200) {$y$};
\end{tikzpicture}
\par\smallskip
$G_{10}$ \quad $|V|=7$, $|E|=12$
\end{minipage}
\par\vspace{3mm}
\begin{minipage}{0.48\textwidth}
\centering
\begin{tikzpicture}[scale=0.95, line width=0.55pt]
\path[use as bounding box] (-1.85,-1.62) rectangle (1.85,1.95);
\coordinate (a) at (0,1.55);
\coordinate (b) at (-1.48,0.48);
\coordinate (c) at (-0.91,-1.26);
\coordinate (d) at (0.91,-1.26);
\coordinate (e) at (1.48,0.48);
\coordinate (x) at (-0.38,0.13);
\coordinate (y) at (0.39,-0.42);
\draw (a)--(b);
\draw (a)--(e);
\draw (a)--(x);
\draw (b)--(c);
\draw (b)--(y);
\draw (c)--(d);
\draw (c)--(x);
\draw (c)--(y);
\draw (d)--(e);
\draw (d)--(x);
\draw (e)--(y);
\draw (x)--(y);
\fill (a) circle (1.5pt) node[above, inner sep=2.5pt] {$a$};
\fill (b) circle (1.5pt) node[left, inner sep=2.5pt] {$b$};
\fill (c) circle (1.5pt) node[below left, inner sep=2.5pt] {$c$};
\fill (d) circle (1.5pt) node[below right, inner sep=2.5pt] {$d$};
\fill (e) circle (1.5pt) node[right, inner sep=2.5pt] {$e$};
\fill (x) circle (1.5pt);
\node[inner sep=0pt] at (-0.6363,0.4229) {$x$};
\fill (y) circle (1.5pt);
\node[inner sep=0pt] at (0.7927,-0.5664) {$y$};
\end{tikzpicture}
\par\smallskip
$G_{11}$ \quad $|V|=7$, $|E|=12$
\end{minipage}
\hfill
\begin{minipage}{0.48\textwidth}
\centering
\begin{tikzpicture}[scale=0.95, line width=0.55pt]
\path[use as bounding box] (-1.85,-1.62) rectangle (1.85,1.95);
\coordinate (a) at (0,1.55);
\coordinate (b) at (-1.48,0.48);
\coordinate (c) at (-0.91,-1.26);
\coordinate (d) at (0.91,-1.26);
\coordinate (e) at (1.48,0.48);
\coordinate (x) at (0,-0.53);
\coordinate (y) at (0,0.25);
\draw (a)--(b);
\draw (a)--(e);
\draw (a)--(y);
\draw (b)--(c);
\draw (b)--(x);
\draw (c)--(d);
\draw (c)--(x);
\draw (c)--(y);
\draw (d)--(e);
\draw (d)--(x);
\draw (d)--(y);
\draw (e)--(x);
\fill (a) circle (1.5pt) node[above, inner sep=2.5pt] {$a$};
\fill (b) circle (1.5pt) node[left, inner sep=2.5pt] {$b$};
\fill (c) circle (1.5pt) node[below left, inner sep=2.5pt] {$c$};
\fill (d) circle (1.5pt) node[below right, inner sep=2.5pt] {$d$};
\fill (e) circle (1.5pt) node[right, inner sep=2.5pt] {$e$};
\fill (x) circle (1.5pt);
\node[inner sep=0pt] at (0.0000,-0.9327) {$x$};
\fill (y) circle (1.5pt);
\node[inner sep=0pt] at (-0.2929,0.3964) {$y$};
\end{tikzpicture}
\par\smallskip
$G_{12}$ \quad $|V|=7$, $|E|=12$
\end{minipage}
\par\vspace{3mm}
\begin{minipage}{0.48\textwidth}
\centering
\begin{tikzpicture}[scale=0.95, line width=0.55pt]
\path[use as bounding box] (-1.85,-1.62) rectangle (1.85,1.95);
\coordinate (a) at (0,1.55);
\coordinate (b) at (-1.48,0.48);
\coordinate (c) at (-0.91,-1.26);
\coordinate (d) at (0.91,-1.26);
\coordinate (e) at (1.48,0.48);
\coordinate (x) at (-0.38,0.13);
\coordinate (y) at (0.39,-0.42);
\draw (a)--(b);
\draw (a)--(e);
\draw (a)--(x);
\draw (a)--(y);
\draw (b)--(c);
\draw (c)--(d);
\draw (c)--(x);
\draw (c)--(y);
\draw (d)--(e);
\draw (d)--(x);
\draw (d)--(y);
\draw (e)--(x);
\draw (x)--(y);
\fill (a) circle (1.5pt) node[above, inner sep=2.5pt] {$a$};
\fill (b) circle (1.5pt) node[left, inner sep=2.5pt] {$b$};
\fill (c) circle (1.5pt) node[below left, inner sep=2.5pt] {$c$};
\fill (d) circle (1.5pt) node[below right, inner sep=2.5pt] {$d$};
\fill (e) circle (1.5pt) node[right, inner sep=2.5pt] {$e$};
\fill (x) circle (1.5pt);
\node[inner sep=0pt] at (-0.6729,0.2398) {$x$};
\fill (y) circle (1.5pt);
\node[inner sep=0pt] at (0.6829,-0.2369) {$y$};
\end{tikzpicture}
\par\smallskip
$G_{13}$ \quad $|V|=7$, $|E|=13$
\end{minipage}
\hfill
\begin{minipage}{0.48\textwidth}
\centering
\begin{tikzpicture}[scale=0.95, line width=0.55pt]
\path[use as bounding box] (-1.85,-1.62) rectangle (1.85,1.95);
\coordinate (a) at (0,1.55);
\coordinate (b) at (-1.48,0.48);
\coordinate (c) at (-0.91,-1.26);
\coordinate (d) at (0.91,-1.26);
\coordinate (e) at (1.48,0.48);
\coordinate (x) at (-0.38,0.13);
\coordinate (y) at (0.39,-0.42);
\draw (a)--(b);
\draw (a)--(e);
\draw (a)--(y);
\draw (b)--(c);
\draw (b)--(e);
\draw (b)--(x);
\draw (c)--(d);
\draw (c)--(x);
\draw (d)--(e);
\draw (d)--(x);
\draw (d)--(y);
\draw (e)--(x);
\draw (x)--(y);
\fill (a) circle (1.5pt) node[above, inner sep=2.5pt] {$a$};
\fill (b) circle (1.5pt) node[left, inner sep=2.5pt] {$b$};
\fill (c) circle (1.5pt) node[below left, inner sep=2.5pt] {$c$};
\fill (d) circle (1.5pt) node[below right, inner sep=2.5pt] {$d$};
\fill (e) circle (1.5pt) node[right, inner sep=2.5pt] {$e$};
\fill (x) circle (1.5pt);
\node[inner sep=0pt] at (-0.8193,-0.0531) {$x$};
\fill (y) circle (1.5pt);
\node[inner sep=0pt] at (0.6829,-0.2369) {$y$};
\end{tikzpicture}
\par\smallskip
$G_{14}$ \quad $|V|=7$, $|E|=13$
\end{minipage}
\par\vspace{3mm}
\begin{minipage}{0.48\textwidth}
\centering
\begin{tikzpicture}[scale=0.95, line width=0.55pt]
\path[use as bounding box] (-1.85,-1.62) rectangle (1.85,1.95);
\coordinate (a) at (0,1.55);
\coordinate (b) at (-1.48,0.48);
\coordinate (c) at (-0.91,-1.26);
\coordinate (d) at (0.91,-1.26);
\coordinate (e) at (1.48,0.48);
\coordinate (x) at (-0.38,0.13);
\coordinate (y) at (0.39,-0.42);
\draw (a)--(b);
\draw (a)--(d);
\draw (a)--(e);
\draw (b)--(c);
\draw (b)--(x);
\draw (b)--(y);
\draw (c)--(d);
\draw (c)--(x);
\draw (d)--(e);
\draw (d)--(x);
\draw (e)--(x);
\draw (e)--(y);
\draw (x)--(y);
\fill (a) circle (1.5pt) node[above, inner sep=2.5pt] {$a$};
\fill (b) circle (1.5pt) node[left, inner sep=2.5pt] {$b$};
\fill (c) circle (1.5pt) node[below left, inner sep=2.5pt] {$c$};
\fill (d) circle (1.5pt) node[below right, inner sep=2.5pt] {$d$};
\fill (e) circle (1.5pt) node[right, inner sep=2.5pt] {$e$};
\fill (x) circle (1.5pt);
\node[inner sep=0pt] at (-0.3800,0.4229) {$x$};
\fill (y) circle (1.5pt);
\node[inner sep=0pt] at (-0.0127,-0.7861) {$y$};
\end{tikzpicture}
\par\smallskip
$G_{15}$ \quad $|V|=7$, $|E|=13$
\end{minipage}
\hfill
\begin{minipage}{0.48\textwidth}
\centering
\begin{tikzpicture}[scale=0.95, line width=0.55pt]
\path[use as bounding box] (-1.85,-1.62) rectangle (1.85,1.95);
\coordinate (a) at (0,1.55);
\coordinate (b) at (-1.48,0.48);
\coordinate (c) at (-0.91,-1.26);
\coordinate (d) at (0.91,-1.26);
\coordinate (e) at (1.48,0.48);
\coordinate (x) at (-0.38,0.13);
\coordinate (y) at (0.39,-0.42);
\draw (a)--(b);
\draw (a)--(e);
\draw (a)--(y);
\draw (b)--(c);
\draw (b)--(x);
\draw (c)--(d);
\draw (c)--(x);
\draw (c)--(y);
\draw (d)--(e);
\draw (d)--(x);
\draw (d)--(y);
\draw (e)--(x);
\draw (x)--(y);
\fill (a) circle (1.5pt) node[above, inner sep=2.5pt] {$a$};
\fill (b) circle (1.5pt) node[left, inner sep=2.5pt] {$b$};
\fill (c) circle (1.5pt) node[below left, inner sep=2.5pt] {$c$};
\fill (d) circle (1.5pt) node[below right, inner sep=2.5pt] {$d$};
\fill (e) circle (1.5pt) node[right, inner sep=2.5pt] {$e$};
\fill (x) circle (1.5pt);
\node[inner sep=0pt] at (-0.3800,0.4229) {$x$};
\fill (y) circle (1.5pt);
\node[inner sep=0pt] at (0.6829,-0.2369) {$y$};
\end{tikzpicture}
\par\smallskip
$G_{16}$ \quad $|V|=7$, $|E|=13$
\end{minipage}
\caption{Graphs $G_{9}$--$G_{16}$ in the Python enumeration.}
\end{figure}
\clearpage

\begin{figure}[p]
\centering
\begin{minipage}{0.48\textwidth}
\centering
\begin{tikzpicture}[scale=0.95, line width=0.55pt]
\path[use as bounding box] (-1.85,-1.62) rectangle (1.85,1.95);
\coordinate (a) at (0,1.55);
\coordinate (b) at (-1.48,0.48);
\coordinate (c) at (-0.91,-1.26);
\coordinate (d) at (0.91,-1.26);
\coordinate (e) at (1.48,0.48);
\coordinate (x) at (-0.38,0.13);
\coordinate (y) at (0.39,-0.42);
\draw (a)--(b);
\draw (a)--(e);
\draw (a)--(x);
\draw (a)--(y);
\draw (b)--(c);
\draw (b)--(d);
\draw (c)--(d);
\draw (c)--(x);
\draw (c)--(y);
\draw (d)--(e);
\draw (d)--(y);
\draw (e)--(x);
\draw (x)--(y);
\fill (a) circle (1.5pt) node[above, inner sep=2.5pt] {$a$};
\fill (b) circle (1.5pt) node[left, inner sep=2.5pt] {$b$};
\fill (c) circle (1.5pt) node[below left, inner sep=2.5pt] {$c$};
\fill (d) circle (1.5pt) node[below right, inner sep=2.5pt] {$d$};
\fill (e) circle (1.5pt) node[right, inner sep=2.5pt] {$e$};
\fill (x) circle (1.5pt);
\node[inner sep=0pt] at (-0.6363,0.3497) {$x$};
\fill (y) circle (1.5pt);
\node[inner sep=0pt] at (0.6829,-0.2369) {$y$};
\end{tikzpicture}
\par\smallskip
$G_{17}$ \quad $|V|=7$, $|E|=13$
\end{minipage}
\hfill
\begin{minipage}{0.48\textwidth}
\centering
\begin{tikzpicture}[scale=0.95, line width=0.55pt]
\path[use as bounding box] (-1.85,-1.62) rectangle (1.85,1.95);
\coordinate (a) at (0,1.55);
\coordinate (b) at (-1.48,0.48);
\coordinate (c) at (-0.91,-1.26);
\coordinate (d) at (0.91,-1.26);
\coordinate (e) at (1.48,0.48);
\coordinate (x) at (-0.38,0.13);
\coordinate (y) at (0.39,-0.42);
\draw (a)--(b);
\draw (a)--(e);
\draw (a)--(y);
\draw (b)--(c);
\draw (b)--(x);
\draw (b)--(y);
\draw (c)--(d);
\draw (c)--(x);
\draw (c)--(y);
\draw (d)--(e);
\draw (d)--(x);
\draw (d)--(y);
\draw (e)--(x);
\fill (a) circle (1.5pt) node[above, inner sep=2.5pt] {$a$};
\fill (b) circle (1.5pt) node[left, inner sep=2.5pt] {$b$};
\fill (c) circle (1.5pt) node[below left, inner sep=2.5pt] {$c$};
\fill (d) circle (1.5pt) node[below right, inner sep=2.5pt] {$d$};
\fill (e) circle (1.5pt) node[right, inner sep=2.5pt] {$e$};
\fill (x) circle (1.5pt);
\node[inner sep=0pt] at (-0.3800,0.4229) {$x$};
\fill (y) circle (1.5pt);
\node[inner sep=0pt] at (0.6829,-0.2369) {$y$};
\end{tikzpicture}
\par\smallskip
$G_{18}$ \quad $|V|=7$, $|E|=13$
\end{minipage}
\par\vspace{3mm}
\begin{minipage}{0.48\textwidth}
\centering
\begin{tikzpicture}[scale=0.95, line width=0.55pt]
\path[use as bounding box] (-1.85,-1.62) rectangle (1.85,1.95);
\coordinate (a) at (0,1.55);
\coordinate (b) at (-1.48,0.48);
\coordinate (c) at (-0.91,-1.26);
\coordinate (d) at (0.91,-1.26);
\coordinate (e) at (1.48,0.48);
\coordinate (x) at (-0.38,0.13);
\coordinate (y) at (0.39,-0.42);
\draw (a)--(b);
\draw (a)--(e);
\draw (a)--(x);
\draw (a)--(y);
\draw (b)--(c);
\draw (b)--(x);
\draw (c)--(d);
\draw (c)--(x);
\draw (c)--(y);
\draw (d)--(e);
\draw (d)--(x);
\draw (d)--(y);
\draw (e)--(x);
\draw (x)--(y);
\fill (a) circle (1.5pt) node[above, inner sep=2.5pt] {$a$};
\fill (b) circle (1.5pt) node[left, inner sep=2.5pt] {$b$};
\fill (c) circle (1.5pt) node[below left, inner sep=2.5pt] {$c$};
\fill (d) circle (1.5pt) node[below right, inner sep=2.5pt] {$d$};
\fill (e) circle (1.5pt) node[right, inner sep=2.5pt] {$e$};
\fill (x) circle (1.5pt);
\node[inner sep=0pt] at (-0.5997,0.4961) {$x$};
\fill (y) circle (1.5pt);
\node[inner sep=0pt] at (0.6829,-0.2369) {$y$};
\end{tikzpicture}
\par\smallskip
$G_{19}$ \quad $|V|=7$, $|E|=14$
\end{minipage}
\hfill
\begin{minipage}{0.48\textwidth}
\centering
\begin{tikzpicture}[scale=0.95, line width=0.55pt]
\path[use as bounding box] (-1.85,-1.62) rectangle (1.85,1.95);
\coordinate (a) at (0,1.55);
\coordinate (b) at (-1.48,0.48);
\coordinate (c) at (-0.91,-1.26);
\coordinate (d) at (0.91,-1.26);
\coordinate (e) at (1.48,0.48);
\coordinate (x) at (-0.38,0.13);
\coordinate (y) at (0.39,-0.42);
\draw (a)--(b);
\draw (a)--(e);
\draw (a)--(x);
\draw (a)--(y);
\draw (b)--(c);
\draw (b)--(x);
\draw (c)--(d);
\draw (c)--(x);
\draw (c)--(y);
\draw (d)--(e);
\draw (d)--(x);
\draw (d)--(y);
\draw (e)--(y);
\draw (x)--(y);
\fill (a) circle (1.5pt) node[above, inner sep=2.5pt] {$a$};
\fill (b) circle (1.5pt) node[left, inner sep=2.5pt] {$b$};
\fill (c) circle (1.5pt) node[below left, inner sep=2.5pt] {$c$};
\fill (d) circle (1.5pt) node[below right, inner sep=2.5pt] {$d$};
\fill (e) circle (1.5pt) node[right, inner sep=2.5pt] {$e$};
\fill (x) circle (1.5pt);
\node[inner sep=0pt] at (-0.5997,0.4961) {$x$};
\fill (y) circle (1.5pt);
\node[inner sep=0pt] at (0.5731,0.2024) {$y$};
\end{tikzpicture}
\par\smallskip
$G_{20}$ \quad $|V|=7$, $|E|=14$
\end{minipage}
\par\vspace{3mm}
\begin{minipage}{0.48\textwidth}
\centering
\begin{tikzpicture}[scale=0.95, line width=0.55pt]
\path[use as bounding box] (-1.85,-1.62) rectangle (1.85,1.95);
\coordinate (a) at (0,1.55);
\coordinate (b) at (-1.48,0.48);
\coordinate (c) at (-0.91,-1.26);
\coordinate (d) at (0.91,-1.26);
\coordinate (e) at (1.48,0.48);
\coordinate (x) at (-0.38,0.13);
\coordinate (y) at (0.39,-0.42);
\draw (a)--(b);
\draw (a)--(e);
\draw (a)--(y);
\draw (b)--(c);
\draw (b)--(x);
\draw (c)--(d);
\draw (c)--(x);
\draw (c)--(y);
\draw (d)--(e);
\draw (d)--(x);
\draw (d)--(y);
\draw (e)--(x);
\draw (e)--(y);
\draw (x)--(y);
\fill (a) circle (1.5pt) node[above, inner sep=2.5pt] {$a$};
\fill (b) circle (1.5pt) node[left, inner sep=2.5pt] {$b$};
\fill (c) circle (1.5pt) node[below left, inner sep=2.5pt] {$c$};
\fill (d) circle (1.5pt) node[below right, inner sep=2.5pt] {$d$};
\fill (e) circle (1.5pt) node[right, inner sep=2.5pt] {$e$};
\fill (x) circle (1.5pt);
\node[inner sep=0pt] at (-0.3800,0.4229) {$x$};
\fill (y) circle (1.5pt);
\node[inner sep=0pt] at (0.8293,-0.4932) {$y$};
\end{tikzpicture}
\par\smallskip
$G_{21}$ \quad $|V|=7$, $|E|=14$
\end{minipage}
\hfill
\begin{minipage}{0.48\textwidth}
\centering
\begin{tikzpicture}[scale=0.95, line width=0.55pt]
\path[use as bounding box] (-1.85,-1.62) rectangle (1.85,1.95);
\coordinate (a) at (0,1.55);
\coordinate (b) at (-1.48,0.48);
\coordinate (c) at (-0.91,-1.26);
\coordinate (d) at (0.91,-1.26);
\coordinate (e) at (1.48,0.48);
\coordinate (x) at (-0.38,0.13);
\coordinate (y) at (0.39,-0.42);
\draw (a)--(b);
\draw (a)--(e);
\draw (a)--(x);
\draw (b)--(c);
\draw (b)--(x);
\draw (b)--(y);
\draw (c)--(d);
\draw (c)--(x);
\draw (c)--(y);
\draw (d)--(e);
\draw (d)--(x);
\draw (d)--(y);
\draw (e)--(x);
\draw (e)--(y);
\fill (a) circle (1.5pt) node[above, inner sep=2.5pt] {$a$};
\fill (b) circle (1.5pt) node[left, inner sep=2.5pt] {$b$};
\fill (c) circle (1.5pt) node[below left, inner sep=2.5pt] {$c$};
\fill (d) circle (1.5pt) node[below right, inner sep=2.5pt] {$d$};
\fill (e) circle (1.5pt) node[right, inner sep=2.5pt] {$e$};
\fill (x) circle (1.5pt);
\node[inner sep=0pt] at (-0.5997,0.4961) {$x$};
\fill (y) circle (1.5pt);
\node[inner sep=0pt] at (0.3534,-0.0173) {$y$};
\end{tikzpicture}
\par\smallskip
$G_{22}$ \quad $|V|=7$, $|E|=14$
\end{minipage}
\par\vspace{3mm}
\begin{minipage}{0.48\textwidth}
\centering
\begin{tikzpicture}[scale=0.95, line width=0.55pt]
\path[use as bounding box] (-1.85,-1.62) rectangle (1.85,1.95);
\coordinate (a) at (0,1.55);
\coordinate (b) at (-1.48,0.48);
\coordinate (c) at (-0.91,-1.26);
\coordinate (d) at (0.91,-1.26);
\coordinate (e) at (1.48,0.48);
\coordinate (x) at (-0.38,0.13);
\coordinate (y) at (0.39,-0.42);
\draw (a)--(b);
\draw (a)--(e);
\draw (a)--(x);
\draw (b)--(c);
\draw (b)--(e);
\draw (b)--(y);
\draw (c)--(d);
\draw (c)--(x);
\draw (c)--(y);
\draw (d)--(e);
\draw (d)--(x);
\draw (d)--(y);
\draw (e)--(y);
\draw (x)--(y);
\fill (a) circle (1.5pt) node[above, inner sep=2.5pt] {$a$};
\fill (b) circle (1.5pt) node[left, inner sep=2.5pt] {$b$};
\fill (c) circle (1.5pt) node[below left, inner sep=2.5pt] {$c$};
\fill (d) circle (1.5pt) node[below right, inner sep=2.5pt] {$d$};
\fill (e) circle (1.5pt) node[right, inner sep=2.5pt] {$e$};
\fill (x) circle (1.5pt);
\node[inner sep=0pt] at (0.0227,0.2398) {$x$};
\fill (y) circle (1.5pt);
\node[inner sep=0pt] at (0.4632,0.1292) {$y$};
\end{tikzpicture}
\par\smallskip
$G_{23}$ \quad $|V|=7$, $|E|=14$
\end{minipage}
\hfill
\begin{minipage}{0.48\textwidth}
\centering
\begin{tikzpicture}[scale=0.95, line width=0.55pt]
\path[use as bounding box] (-1.85,-1.62) rectangle (1.85,1.95);
\coordinate (a) at (0,1.55);
\coordinate (b) at (-1.48,0.48);
\coordinate (c) at (-0.91,-1.26);
\coordinate (d) at (0.91,-1.26);
\coordinate (e) at (1.48,0.48);
\coordinate (x) at (-0.38,0.13);
\coordinate (y) at (0.39,-0.42);
\draw (a)--(b);
\draw (a)--(d);
\draw (a)--(e);
\draw (a)--(y);
\draw (b)--(c);
\draw (b)--(x);
\draw (b)--(y);
\draw (c)--(d);
\draw (c)--(x);
\draw (c)--(y);
\draw (d)--(e);
\draw (d)--(x);
\draw (e)--(x);
\draw (e)--(y);
\fill (a) circle (1.5pt) node[above, inner sep=2.5pt] {$a$};
\fill (b) circle (1.5pt) node[left, inner sep=2.5pt] {$b$};
\fill (c) circle (1.5pt) node[below left, inner sep=2.5pt] {$c$};
\fill (d) circle (1.5pt) node[below right, inner sep=2.5pt] {$d$};
\fill (e) circle (1.5pt) node[right, inner sep=2.5pt] {$e$};
\fill (x) circle (1.5pt);
\node[inner sep=0pt] at (-0.3800,0.4229) {$x$};
\fill (y) circle (1.5pt);
\node[inner sep=0pt] at (0.1703,-0.9325) {$y$};
\end{tikzpicture}
\par\smallskip
$G_{24}$ \quad $|V|=7$, $|E|=14$
\end{minipage}
\caption{Graphs $G_{17}$--$G_{24}$ in the Python enumeration.}
\end{figure}
\clearpage

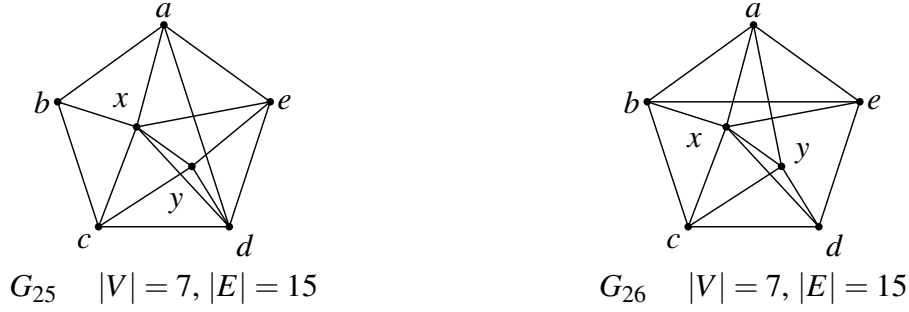
\begin{figure}[!ht]
\centering
\begin{minipage}{0.48\textwidth}
\centering
\begin{tikzpicture}[scale=0.95, line width=0.55pt]
\path[use as bounding box] (-1.85,-1.62) rectangle (1.85,1.95);
\coordinate (a) at (0,1.55);
\coordinate (b) at (-1.48,0.48);
\coordinate (c) at (-0.91,-1.26);
\coordinate (d) at (0.91,-1.26);
\coordinate (e) at (1.48,0.48);
\coordinate (x) at (-0.38,0.13);
\coordinate (y) at (0.39,-0.42);
\draw (a)--(b);
\draw (a)--(d);
\draw (a)--(e);
\draw (a)--(x);
\draw (b)--(c);
\draw (b)--(x);
\draw (c)--(d);
\draw (c)--(x);
\draw (c)--(y);
\draw (d)--(e);
\draw (d)--(x);
\draw (d)--(y);
\draw (e)--(x);
\draw (e)--(y);
\draw (x)--(y);
\fill (a) circle (1.5pt) node[above, inner sep=2.5pt] {$a$};
\fill (b) circle (1.5pt) node[left, inner sep=2.5pt] {$b$};
\fill (c) circle (1.5pt) node[below left, inner sep=2.5pt] {$c$};
\fill (d) circle (1.5pt) node[below right, inner sep=2.5pt] {$d$};
\fill (e) circle (1.5pt) node[right, inner sep=2.5pt] {$e$};
\fill (x) circle (1.5pt);
\node[inner sep=0pt] at (-0.5997,0.4961) {$x$};
\fill (y) circle (1.5pt);
\node[inner sep=0pt] at (0.1703,-0.9325) {$y$};
\end{tikzpicture}
\par\smallskip
$G_{25}$ \quad $|V|=7$, $|E|=15$
\end{minipage}
\hfill
\begin{minipage}{0.48\textwidth}
\centering
\begin{tikzpicture}[scale=0.95, line width=0.55pt]
\path[use as bounding box] (-1.85,-1.62) rectangle (1.85,1.95);
\coordinate (a) at (0,1.55);
\coordinate (b) at (-1.48,0.48);
\coordinate (c) at (-0.91,-1.26);
\coordinate (d) at (0.91,-1.26);
\coordinate (e) at (1.48,0.48);
\coordinate (x) at (-0.38,0.13);
\coordinate (y) at (0.39,-0.42);
\draw (a)--(b);
\draw (a)--(e);
\draw (a)--(x);
\draw (a)--(y);
\draw (b)--(c);
\draw (b)--(e);
\draw (b)--(x);
\draw (c)--(d);
\draw (c)--(x);
\draw (c)--(y);
\draw (d)--(e);
\draw (d)--(x);
\draw (d)--(y);
\draw (e)--(x);
\draw (x)--(y);
\fill (a) circle (1.5pt) node[above, inner sep=2.5pt] {$a$};
\fill (b) circle (1.5pt) node[left, inner sep=2.5pt] {$b$};
\fill (c) circle (1.5pt) node[below left, inner sep=2.5pt] {$c$};
\fill (d) circle (1.5pt) node[below right, inner sep=2.5pt] {$d$};
\fill (e) circle (1.5pt) node[right, inner sep=2.5pt] {$e$};
\fill (x) circle (1.5pt);
\node[inner sep=0pt] at (-0.8193,-0.0531) {$x$};
\fill (y) circle (1.5pt);
\node[inner sep=0pt] at (0.6829,-0.2369) {$y$};
\end{tikzpicture}
\par\smallskip
$G_{26}$ \quad $|V|=7$, $|E|=15$
\end{minipage}
\caption{Graphs $G_{25}$--$G_{26}$ in the Python enumeration.}
\end{figure}

\end{document}